\documentclass[article]{amsart}

\usepackage{CJK,CJKnumb}
\usepackage{color}
\usepackage{indentfirst}
\usepackage{latexsym,bm}
\usepackage{amsmath,amssymb}
\usepackage{graphicx}
\usepackage{bbm}
\usepackage{amsthm}
\usepackage[ansinew]{inputenc}
\usepackage{array}
\usepackage{amsxtra}
\usepackage{amstext}
\usepackage{latexsym}
\usepackage{amsfonts}
\usepackage{bm}
\usepackage[pdftex,pdfstartview=FitH,bookmarksnumbered,bookmarksopen,bookmarksopenlevel=2,CJKbookmarks]{hyperref}

\newtheorem{theorem}{Theorem}[section]
\newtheorem{definition}[theorem]{Definition}

\newtheorem{proposition}[theorem]{Proposition}

\newtheorem{lemma}[theorem]{Lemma}
\newcommand{\lmref}[1]{Lemma \ref{#1}}
\newcommand{\thmref}[1]{Theorem \ref{#1}}

\newcommand{\propref}[1]{Proposition \ref{#1}}

\newcommand{\defref}[1]{Definition \ref{#1}}
\newcommand{\secref}[1]{Section \ref{#1}}

\numberwithin{equation}{section}

\begin{document}

\title{Magnetohydrodynamic regime of the Born-Infeld electromagnetism}
\author{
%\\
Xianglong Duan}
\date\today



\address{
CNRS UMR 7640 \\ % \hfill (Received 00 00 2010)\\
Ecole Polytechnique   \\ %\hfill (Revised  00 00 2010)\\
Palaiseau\\
France}

\email{xianglong.duan@polytechnique.edu}

\begin{abstract}
The Born-Infeld (BI) model is a nonlinear correction of Maxwell's equations. By adding the energy and Poynting vector as additional variables, it can be augmented as a $10\times10$ system of hyperbolic conservation laws, called the augmented BI (ABI) equations. The author found that, through a quadratic change of the time variable, the ABI system gives a simple energy dissipation model that combines Darcy's law and magnetohydrodynamics (MHD). Using the concept of ``relative entropy" (or ``modulated energy"), borrowed from the theory
of hyperbolic systems of conservation laws, we introduce a notion of generalized solutions, that we call dissipative solutions. For given initial conditions, the set of generalized solutions is not empty, convex, and compact. Smooth solutions to the dissipative system are always unique in this setting.
%(?????????????????????????)
\end{abstract}
\maketitle

%%%%%%%%%%%%%%%%%%%%%%%%%%%%%%%%%%%%%%%%%%
%%%%%%%%%%%%%%%%%%%%%%%%%%%%%%%%%%%%%%%%%%
%%%%%%%%%%%%%%%%%%%%%%%%%%%%%%%%%%%%%%%%%%

\section{Introduction}

There are many examples of dissipative systems that can be derived from conservative ones. The derivation can be
done in many different ways, for example by adding a very strong friction term or by homogenization techniques
% \cite{Marcati,Allaire,StRaymond...}
or by properly rescaling the time variable by a small
parameter (through the so-called ``parabolic scaling"). In the recent work of the author and Y. Brenier \cite{BD}, we suggested a very straightforward idea: just perform the
quadratic change of time
$t\rightarrow \theta=t^2/2$.
Several examples were studied in that paper. One example was the porous medium equation, which can be retrieved from the Euler equation of isentropic gases. Another relevant example, at the interface of Geometry and High Energy Physics, is the dissipative geometric model of curve-shortening flow in
$\mathbb{R}^d$ (which is the simplest example of mean-curvature flow with co-dimension higher than $1$) that we obtained from the
conservative evolution of classical strings according to the Nambu-Goto action. This paper is a follow-up of \cite{BD}, where the Born-Infeld model of Electromagnetism is taken as an example, and as a result, we get a dissipative model of Magnetohydrodynamics (MHD) where we have non-linear diffusions in the magnetic induction equation and the Darcy's law for the velocity field.

The Born-Infeld (BI) equations were originally introduced by Max Born and Leopold Infeld in 1934 \cite{BI} as a nonlinear correction to the linear Maxwell equations allowing finite electrostatic fields for point charges. In high energy Physics, D-branes can be modelled according to a generalization of the BI model \cite{Po,Gi}. In differential geometry, the BI equations are closely related to the study of extremal surfaces in the Minkowski space. In the 4-dimensional Minkowski space of special relativity, the BI equations form a $6\times 6$ system of conservation laws
in the sense of \cite{Dafer}, with $2$ differential constraints,
$$
\partial_t B+\nabla\times
\left(\frac{B\times(D\times B)+D}{h}\right)=0,
\;\;\;\partial_t D+\nabla\times
\left(\frac{D\times(D\times B)-B}{h}\right)=0,
$$
$$
h=\sqrt{1+D^2+B^2+(D\times B)^2},
\;\;\;\nabla\cdot B=\nabla\cdot D=0,
$$
where we use the conventional notations for the inner product $\cdot$ and the cross-product $\times$ in $\mathbb{R}^3$,
the gradient operator $\nabla$,
the curl operator $\nabla\times$ and the electromagnetic field $(B,D)$.
By Noether's theorem, this system admits 4 extra conservation laws for the energy density $h$ and Poynting vector $P$, namely,
$$
\partial_t h+\nabla\cdot P=0,\;\;\;
\partial_t P+
\nabla\cdot\left(\frac{P\otimes P-B\otimes B-D\otimes D}{h}\right)=\nabla \left(\frac{1}{h}\right),
$$
where
$$
P=D\times B,\;\;
h=\sqrt{1+D^2+B^2+|D\times B|^2}.
$$
As advocated in \cite{Br}, by viewing $h,P$ as independent variables, the BI system can be ``augmented'' as a $10\times 10$ system of hyperbolic conservation laws with an extra conservation law involving a ``strictly convex'' entropy, namely
$$h^{-1}(1+B^2+D^2+P^2).$$
This augmented BI system belongs to the
nice class of systems of conservation laws ``with convex entropy'', which, under secondary suitable additional conditions, enjoy
important properties such as well-posedness of the initial value problem, at least for short times, and ``weak-strong'' uniqueness principles \cite{Dafer}.

For the $10\times10$ augmented BI system, we obtain, after the quadratic change of the time variable $t\rightarrow\theta=t^2/2$,
the following asymptotic system as $\theta<<1$:
$$
\partial_\theta B+\nabla\times (h^{-1}B\times P)
+\nabla\times (h^{-1}\nabla\times(h^{-1}B))=0,
$$
$$
\partial_\theta h+\nabla\cdot P=0,\;\;\;
P=
\nabla\cdot(h^{-1}B\otimes B)+\nabla (h^{-1}).
$$
This system can be interpreted
as an unusual, fully dissipative version of standard Magnetohydrodynamics, including a generalized
version of the Darcy law, with a
fluid of density $h$, momentum $P$ and pressure $p=-h^{-1}$ (of Chaplygin type), interacting with a magnetic field $B$. It belongs to the class of non-linear degenerate parabolic PDEs.

In the rest of the paper, we proceed to the analysis of this asymptotic model (that we call ``Darcy MHD'') obtained after rescaling the $10\times 10$
augmented BI model: (i) in \secref{sec:3},
we define a concept of ``dissipative solutions''
in a sense inspired by the work of P.-L. Lions for the Euler equation of incompressible fluids \cite{Li},
the work of L. Ambrosio, N. Gigli, G. Savar\'e \cite{AGS} for the heat equation (working in a very general
class of metric measured spaces) and quite similar to the one recently introduced by Y. Brenier in \cite{Br-CMP}; (ii) in \secref{sec:4}, we demonstrate some properties of the dissipative solutions. we establish a ``weak-strong'' uniqueness principle, in the sense that,
for a fixed smooth initial condition, a smooth classical solutions is necessarily unique in the class
of dissipative solutions admitting the same initial condition;
we prove the "weak compactness" of such solutions (i.e. any sequence of dissipative solutions has accumulations
points, in a suitable weak sense, and each of them is still a dissipative solution); (iii) in \secref{sec:4}, we estimate the error between dissipative solutions of the asymptotic system and smooth solutions of the $10\time 10$ augmented Born-Infeld system; (iv) we finally prove the global existence solution of dissipative solution for any initial condition, without any
smoothness assumption. This last point,  which is a non-surprising consequence of the
weak compactness, nevertheless requires a lengthy and
technical proof which is presented in \secref{sec:6} and  \secref{sec:7}.

%%%%%%%%%%%%%%%%%%%%%%%%%%%%%%%%%%%%%

\subsection*{Acknowlegment}
The author is very grateful to his PhD advisor, Yann Brenier for his help during the
completion to this paper.
The author is also very grateful to Alexis F. Vasseur for pointing out references \cite{f1,f2}.
%(???????????????????????????)

%%%%%%%%%%%%%%%%%%%%%%%%%%%%%%%%%%%%%
%%%%%%%%%%%%%%%%%%%%%%%%%%%%%%%%%%%%%
%%%%%%%%%%%%%%%%%%%%%%%%%%%%%%%%%%%%%

\section{Direct Derivation of the diffusion equations}\label{sec:2}

%%%%%%%%%%%%%%%%%%%%%%%%%%%%%%%%%%%%%

\subsection{Presentation of the Born-Infeld model}

For a $n+1$ dimensional spacetime, the Born-Infeld equations can be obtained by varying the Lagrangian of the following density
$$
\mathcal{L}_{BI} = \lambda^2\left(1-\sqrt{-\det\left(\eta + \frac{F}{\lambda}\right)}\right)
$$
where $\eta={\rm diag}(-1,1,\ldots,1)$ is the Minkowski metric tensor, $F_{\mu\nu}=\partial_{\mu}A_{\nu}-\partial_{\nu}A_{\mu}$ is the electromagnetic field tensor with $A$ a vector potential. The parameter $\lambda\in(0,\infty)$ is called the absolute field constant which can be comprehended as the upper limit of the field strength \cite{BI}.
%\cite{BIc}.
In the $4$ dimensional spacetime, by using the classical electromagnetic field symbols $B,D$, the BI equations can be written as
\begin{equation}
\partial_t B + \nabla\times\left(\frac{\lambda^2 D + B\times(D\times B)}{\sqrt{\lambda^4 + \lambda^2B^2+ \lambda^2D^2+|D\times B|^2}}\right)=0,\quad\nabla\cdot B=0,
\end{equation}
\begin{equation}
\partial_t D +\nabla\times\left(\frac{-\lambda^2 B + D\times(D\times B)}{\sqrt{\lambda^4 + \lambda^2B^2+ \lambda^2D^2+|D\times B|^2}}\right),\quad\nabla\cdot D=0.
\end{equation}

Now, let us introduce some background of the BI model. The BI model was originally introduced by Max Born and Leopold Infeld in 1934 \cite{BI} as a nonlinear correction to the linear Maxwell model. Born had already postulated \cite{Born} a universal bound $\lambda$ for any electrostatic field, even generated by a point charge (which is obviously not the case of the Maxwell theory for which the corresponding field is unbounded and not even locally square integrable in space), just as the speed of light is a universal bound for any velocity in special relativity. As $\lambda\rightarrow\infty$ the linear Maxwell theory is easily recovered as an approximation of the BI model. Max Born proposed a precise value for $\lambda$ (based on the mass of the electron) and showed no substantial difference with the Maxwell model until subatomic scales are reached. In this way, the BI model was thought as an alternative to the Maxwell theory to tackle the delicate issue of establishing a consistent quantization of Electromagnetism with $\lambda$ playing the role of a cut-off parameter. As a matter of fact, the BI model rapidly became obsolete for such a purpose, after the arising of Quantum Electrodynamics (QED), where renormalization techniques were able to cure the problems posed by the unboundedness of the Maxwell field generated by point charges. [Interestingly enough, M. Kiessling has recently revisited QED from
a Born-Infeld perspective  \cite{Kiessling,Kiessling-2}.]
Later on, there has been a renewed interest for the BI model in high energy Physics, starting in the 1960s for the modelling of hadrons, with a strong revival in the 1990s, in String Theory. In particular the new concept of D-brane was modelled according to a generalization of the BI model \cite{Po,Gi}.

Another important feature of the BI model is its deep link with differential geometry, already studied in a memoir of the Institut Henri Poincare by Max Born in 1938 \cite{Bo-IHP}. Indeed, the BI equations are closely related to the concept of extremal
surfaces in the Minkowski space. As a matter of fact \cite{Br, Br-W}, as $\lambda\rightarrow 0$, the BI model
provides a faithful description of a continuum of classical strings, which are nothing but extremal surfaces moving in the Minkowski space.

From a PDE viewpoint, the BI equations belong to the family of nonlinear systems of hyperbolic conservation laws \cite{Dafer}, for which the existence and uniqueness of local
in time smooth solutions can be proven by standard devices. A rather impressive result was recently established by J. Speck \cite{Speck} who was able to show the global existence of smooth localized solutions for the original BI system, provided the initial conditions are of small enough amplitude. His proof relies on the null-form method developed by Klainerman and collaborators (in particular for the Einstein
equation) combined with dispersive (Strichartz) estimates. This followed an earlier work of Lindblad on the model of extremal surfaces in the Minkowski space which
can be seen as a ``scalar'' version of the BI system \cite{Lind}.

%%%%%%%%%%%%%%%%%%%%%%%%%%%%%%%%%%%%%%%%%%%%%

\subsection{The $10\times 10$ augmented BI system}
In 2004, Y. Brenier showed that the structure of the BI system can be widely ``simplified'' by using the extra conservation laws of energy and momentum provided by the Noether invariance theorem, where the momentum (called Poynting vector) is $P=D\times B$ while the energy density is $h=\sqrt{1+B^2+D^2+P^2}$ \cite{Br}. They read (after $\lambda$ has been
normalized to be 1, which is possible by a suitable change of physical units)
\begin{equation}
\partial_t h + \nabla\cdot P=0,\;\;\;\partial_t P + \left(\frac{P\otimes P - B\otimes B - D\otimes D}{h}\right)=\nabla\left(\frac{1}{h}\right),
\end{equation}
At this point, there are two main possibilities. The first one amounts to add the conservation of momentum (i.e. 3 additional conservation laws) to the $6\times 6$ original
BI equations, written in a suitable way, where $P$ is considered as independent from
$B$ and $D$ (namely not given by the algebraic relation $P=D\times B$) while $h$ is still $h=\sqrt{1+B^2+D^2+P^2}$. This strategy leads to the $9\times 9$ system and the conservation of energy then reads
$$\partial_t \sqrt{1+B^2+D^2+P^2} + \nabla\cdot P +\nabla\cdot\left(\frac{(D\cdot P)D+(B\cdot P)B-P+D\times B}{1+B^2+D^2+P^2}\right)=0$$
where the energy is now a strictly convex function of $B$, $D$ and $P$. It can be shown \cite{Br} that the algebraic constraint $P=D\times B$ is preserved during the evolution of any smooth solution of this system, which implies that, at least for smooth solutions,
the $9\times 9$ augmented system is perfectly suitable for the analysis of the BI equations. This idea has been successfully extended to a very large class of nonlinear systems in
Electromagnetism by D. Serre \cite{Serre}. An even more radical strategy was followed and emphasized in \cite{Br}, where $h$ itself is considered as a new unknown variable,
independent from $B$, $D$ and $P$, while adding the conservation of both energy and
momentum (i.e. 4 conservation laws) to the original $6\times 6$ BI system, written in
a suitable way. This leads to the following $10\times 10$ system of conservation law for
$B,D,P,h$:
\begin{equation}\label{eq:ABI-1}
\partial_t h + \nabla\cdot P=0,
\;\;\partial_t B+ \nabla\times\left(\frac{B\times P+D}{h}\right)=0,\;\;\nabla\cdot B=\nabla\cdot D=0,
\end{equation}
\begin{equation}\label{eq:ABI-2}
\partial_t D+\nabla\times \left(\frac{D\times P-B}{h}\right)=0,
\;\partial_t P+
\nabla\cdot\left(\frac{P\otimes P - B\otimes B - D\otimes D-I_3}{h}\right)=0,
\end{equation}
Once again, the algebraic constraints, namely
$$P=D\times B,\;\;\;h=\sqrt{1+B^2+D^2+P^2}$$
are preserved during the evolution of smooth solutions. The $10\times 10$ extension has a very nice structure, enjoying invariance under Galilean transforms
$$(t,x,B,D,P,h)\longrightarrow (t,x+Vt,B,D,P-Vh,h)$$
(where $V\in\mathbb{R}^3$ is any fixed constant velocity). This is quite surprising, since the BI model is definitely Lorentzian and not Galilean, but not contradictory since such
Galilean transforms are not compatible with the algebraic constraints:
$$P=D\times B,\;\;\;h=\sqrt{1+B^2+D^2+P^2}$$
In \cite{Br-W} it is further observed that, written in non conservation forms, for variables
$$(\tau,b,d,v)=(1/h,B/h,D/h,P/h)\in\mathbb{R}^{10},$$
the $10\times 10$ system reduces to
\begin{equation}\label{eq:abi-non1}
\partial_t b +(v\cdot\nabla)b=(b\cdot\nabla)v-\tau\nabla\times d,\;\;\partial_t d +(v\cdot\nabla)d=(d\cdot\nabla)v+\tau\nabla\times b,
\end{equation}
\begin{equation}\label{eq:abi-non2}
\partial_t \tau +(v\cdot\nabla)\tau=\tau\nabla\cdot v,\;\;\partial_t v +(v\cdot\nabla)v=(b\cdot\nabla)b+(d\cdot\nabla)d+\tau\nabla\tau,
\end{equation}
which is just a symmetric quadratic system of first order PDEs, automatically well-posed (for short times) in Sobolev spaces, such as $W^{s,2}$ for any $s>5/2$, without any restriction on the values of $(b,d,v,\tau)$ in $\mathbb{R}^{10}$ (including negative values of $\tau$!). Once again, the algebraic constraints, which can be now nicely written as
$$b^2+d^2+v^2+\tau^2=1,\;\;\;\tau v=d\times b$$
are preserved during the evolution. Notice that two interesting reductions of this system can be performed. First, it is consistent to set simultaneously $\tau=0$ and
$d=0$ in the equations, which leads to
$$\partial_t b +(v\cdot\nabla)b=(b\cdot\nabla)v,\;\;\partial_t v +(v\cdot\nabla)v=(b\cdot\nabla)b,$$
while the algebraic constraints become
$$b^2+v^2=1,\;\;\;b\cdot v=0.$$
This system can be used to describe the evolution of a continuum of classical strings (i.e. extremal $2-$surfaces in the $4-$dimensional Minkowski space) \cite{Br-W}. A second
reduction can be obtained by setting $\tau=0$, $b=d=0$ which leads to the inviscid Burgers equation
$$\partial_t v +(v\cdot\nabla)v=0$$
This equation, as well known, always leads to finite time singularity for all smooth localized initial conditions $v$, except for the trivial one: $v=0$ (which, by the way,
shows that Speck's result cannot be extended to the $10\times 10$ BI system, without restrictions on the initial conditions).

%%%%%%%%%%%%%%%%%%%%%%%%%%%%%%%%%%%%%%%%%%%%%

\subsection{Quadratic time rescaling of the augmented BI system}

Let us perform the following rescaling of the $10\times 10$ augmented BI system \eqref{eq:ABI-1}-\eqref{eq:ABI-2}:
$$t\rightarrow\theta=t^2/2,\;\;\;h,B,P,D\rightarrow h,B,P\frac{d\theta}{dt},D\frac{d\theta}{dt}$$
Observe that the symmetry between $B$ and $D$ is broken in this rescaling since $D$ is rescaled in the same way as $P$ but not as $B$. We obtain, after very simple calculations, the following rescaled equations,
$$
\partial_\theta h + \nabla\cdot P=0,
\;\;\partial_\theta B + \nabla\times \left(\frac{B\times P+D}{h}\right)=0,
$$
$$
D+2\theta\left[\partial_\theta D+\nabla\times
\left(\frac{D\times P}{h}\right)\right]=\nabla\times
\left(\frac{B}{h}\right),
$$
$$
P+
2\theta
\left[\partial_\theta P+\nabla\cdot\left(\frac{P\otimes P-D\otimes D}{h}\right)\right]=
\nabla\cdot\left(\frac{B\otimes B}{h}\right)+\nabla (h^{-1}).
$$
In the regime $\theta>>1$, we get a self-consistent system for $(D,P,h)$ (without $B$!)
$$\partial_\theta h + \nabla\cdot P=0,\;\;\;\partial_\theta D+\nabla\times
\left(\frac{D\times P}{h}\right)=0,$$
$$\partial_\theta P+\nabla\cdot\left(\frac{P\otimes P-D\otimes D}{h}\right)=0,$$
which, written in non-conservative variables $(d,v)=(D/h,P/h)$, reduces to
$$\partial_t v +(v\cdot\nabla)v=(d\cdot\nabla)d,\;\;\;\partial_t d +(v\cdot\nabla)d=(d\cdot\nabla)v,$$
that we already saw in the previous subsection as a possible reduction of the ($10\times 10$) extended BI system (which describes the motion of a continuum of strings). The regime of higher interest for us is the dissipative one obtained as $\theta<<1$.
Neglecting the higher order terms as $\theta<<1$, we first get
$$D=\nabla\times
\left(h^{-1}B\right)$$
which allows us to eliminate $D$ and get for $(B,P,h)$ the self-consistent system
$$\partial_\theta B +\nabla\times\left(h^{-1}B\times P\right) + \nabla\times\left(h^{-1}\nabla\times
\left(h^{-1}B\right)\right)=0,$$
$$\partial_\theta h + \nabla\cdot P=0,\;\;\;P=
\nabla\cdot\left(h^{-1}B\otimes B\right)+\nabla \left(h^{-1}\right).$$
This can be seen as a dissipative model of Magnetohydrodynamics (MHD) where a fluid of density $h$ and momentum $P$ interacts with a magnetic field $B$, with several
interesting (and intriguing) features:\\
(i) the first equation, which can be interpreted in MHD terms as the ``induction equation'' for $B$, involves a second-order diffusion term typical of MHD: $\nabla\times\left(h^{-1}\nabla\times
\left(h^{-1}B\right)\right)$ (with, however, an unusual dependence on $h$);
(ii) the third equation describes the motion of the fluid of density $h$ and momentum $P$ driven by the magnetic field $B$ and can be interpreted as a (generalized) Darcy
law (and not as the usual momentum equation of MHD), just if the fluid was moving in a porous medium (which seems highly unusual in MHD!);
(iii) there are many coefficients which depend on $h$ in a very peculiar way; in particular the Darcy law involves the so-called Chaplygin pressure $p=-h^{-1}$ (with sound
speed $\sqrt{dp/dh}=h^{-1}$), which is sometimes used for the modeling of granular flows and also in cosmology, but not (to the best of our knowledge) in standard MHD.

To conclude this subsection, let us emphasize the remarkable structure of the ($10\times 10$) extended Born-Infeld system, after quadratic time-rescaling $t\rightarrow\theta=t^2/2$,
which interpolates between the description of a continuum of strings (as $\theta>>1$), in the style of high energy physics (however without any quantum feature) and a much more ``down to earth'' (but highly conjectural) dissipative model of MHD in
a porous medium (as $\theta<<1$)!

%%%%%%%%%%%%%%%%%%%%%%%%%%%%%%%%%%%%%%%%%%
%%%%%%%%%%%%%%%%%%%%%%%%%%%%%%%%%%%%%%%%%%
%%%%%%%%%%%%%%%%%%%%%%%%%%%%%%%%%%%%%%%%%%

\section{dissipative solution of the diffusion equations}\label{sec:3}

From now on, we focus on the analysis of the following system of diffusion equations (we call Darcy MHD, or DMHD),
\begin{equation}\label{eq:dd-h}
\partial_{t}h +\nabla\cdot \left(h v\right)=0,
\end{equation}
\begin{equation}\label{eq:dd-B}
\partial_{t}B + \nabla\times\left(B\times v + d\right)=0,
\end{equation}
\begin{equation}\label{eq:dd-con}
D=h d = \nabla\times \left(\frac{B}{h}\right),\;\;P=h v = \nabla\cdot \left(\frac{B\otimes B}{h}\right) + \nabla\left(h^{-1}\right),
\end{equation}
\begin{equation}\label{eq:dd-div}
\nabla\cdot B=0.
\end{equation}
Written in the non-conservative variables $(\tau,b,d,v)=(1/h,B/h,D/h,P/h)$, the equation reads
\begin{equation}
\partial_t\tau + v\cdot\nabla\tau =\tau \nabla\cdot v,
\;\;\;\partial_t b + (v\cdot\nabla) b =(b\cdot\nabla)v - \tau\nabla\times d,
\end{equation}
\begin{equation}
d=\tau\nabla\times b,\;\;\;v=(b\cdot\nabla)b + \tau\nabla\tau.
\end{equation}

For simplicity, we consider the periodic solutions on  $[0,T]\times\mathbb{T}^3,\;T>0,\;\mathbb{T}=\mathbb{R}/\mathbb{Z}$.

%%%%%%%%%%%%%%%%%%%%%%%%%%%%%%%%%%%%%%%%%%%%%%%%%%

\subsection{Relative entropy and the idea of dissipative solution}

For the moment, ignoring the existence and regularity issues, we assume that there exists a sufficiently smooth solution $(h>0,B,D,P)$ of the Darcy MHD \eqref{eq:dd-h}-\eqref{eq:dd-div}.

First, as introduced in the previous section, the augmented BI equations \eqref{eq:ABI-1}-\eqref{eq:ABI-2} have a strictly convex entropy, namely,
$$\frac{1+B^2+D^2+P^2}{2h}.$$
By performing the quadratic change of time $t\rightarrow\theta=t^2/2$, in the regime $\theta<<1$, the entropy is reduced to
$$\frac{1+B^2}{2h}.$$
It is natural to consider the above energy for the reduced parabolic system i.e., Darcy MHD. As an easy exercise, we can show that the energy we suggested above is decreasing as time goes on. In fact, we have the following equality,
\begin{equation}\label{eq:dde1}
\frac{{\rm d}}{{\rm d}t}\int_{x\in\mathbb{T}^3}\frac{B^2+1}{2h} + \int_{x\in\mathbb{T}^3}\frac{D^2+P^2}{h} = 0
\end{equation}
(This is easy to check, since
\begin{equation*}
\begin{array}{r@{}l}
\displaystyle{\frac{{\rm d}}{{\rm d}t}\int\frac{B^2+1}{2h}} & = \displaystyle{\int\frac{B\cdot\partial_{t}B}{h}-\int\frac{B^2+1}{2h^2}\partial_{t}h}\\
& = \displaystyle{-\int\nabla\times\left(\frac{B}{h}\right)\cdot\left(B\times v + d\right)-\int\nabla\left(\frac{B^2+1}{2h^2}\right)\cdot P}\\
& = \displaystyle{-\int\left[\nabla\times\left(\frac{B}{h}\right)\right]\cdot\frac{D}{h}-\int\left[\nabla\cdot\left(\frac{B\otimes B+ I_{3}}{h}\right)\right]\cdot \frac{P}{h} }\\
\end{array}
\end{equation*}
which, by \eqref{eq:dd-con}, gives the dissipative term.)

Now, for any smooth test functions $(b^{*},h^{*})\in \mathbb{R}^3\times\mathbb{R}^{+}$, the relative entropy is defined by
$$
\frac{1}{2h}\left[(B-hb^{*})^2+(1-h{h^{*}}^{-1})^2\right]
$$
Before going on, let's look at the following lemma which gives us a nice formula for the relative entropy:
\begin{lemma}\label{lm:dd-entropy}
For any functions $P,B,D,v^*,b^*,d^*\in C^1([0,T]\times\mathbb{T}^3,\mathbb{R}^3)$, and positive functions $0<h,h^*\in C^{1}([0,T]\times\mathbb{T}^3,\mathbb{R})$, suppose $(h,B,D,P)$ is a solution of the Darcy MHD \eqref{eq:dd-h}-\eqref{eq:dd-div}, then the following equality always holds
\begin{equation}\label{eq:dd-rel-1}
\frac{{\rm d}}{{\rm d}t}\int_{x\in\mathbb{T}^3}\frac{\big|\widetilde{U}\big|^2}{2h} +\int_{x\in\mathbb{T}^3}\frac{\widetilde{W}^{\rm T}Q(w^*)\widetilde{W}}{2h} +\int_{x\in\mathbb{T}^3} \widetilde{W}\cdot\mathrm{L}(w^*)=0
\end{equation}
where
$$\widetilde{U}=\left(1-h{h^*}^{-1},B-hb^*\right), \;\;\widetilde{W}=\left(\widetilde{U},D-hd^{*},P-hv^{*}\right),\;\;w^*=({h^*}^{-1},b^*,d^*,v^*),$$
$Q(w^*)$ is a symmetric matrix that has the following expression
\begin{equation}\label{eq:ddm-1}
Q(w^*)=\left(
                       \begin{array}{cccc}
                         -2\nabla\cdot v^{*} &(\nabla\times d^{*})^{{\rm T}} & -(\nabla\times b^{*})^{{\rm T}}& 0 \\
                         \nabla\times d^{*} & -\nabla v^{*}-\nabla {v^{*}}^{{\rm T}} & 0 &\nabla b^{*}-\nabla {b^{*}}^{{\rm T}}\\
                        -\nabla\times b^{*} & 0 & 2I_{3}& 0\\
                         0 & \nabla {b^{*}}^{{\rm T}}-\nabla b^{*} &0 & 2I_{3}\\
                       \end{array}
                     \right),
\end{equation}
$\mathrm{L}(w^*)=\big(\mathrm{L}_{h}(w^*),\mathrm{L}_{B}(w^*),\mathrm{L}_{D}(w^*),\mathrm{L}_{P}(w^*)\big)$ has the following expression
\begin{equation}\label{eq:L-h}
\mathrm{L}_{h}(w^*)=\partial_{t}\big(
{h^*}^{-1}\big) -{h^*}^{-1}\nabla\cdot v^* + v^*\cdot\nabla\big(
{h^*}^{-1}\big),
\end{equation}
\begin{equation}
\mathrm{L}_{B}(w^*)=\partial_{t}b^* + (v^*\cdot\nabla)b^*-(b^*\cdot\nabla)v^*+{h^*}^{-1}
\nabla\times d^*,
\end{equation}
\begin{equation}
\mathrm{L}_{D}(w^*)=d^*-{h^*}^{-1}
\nabla\times b^*,
\end{equation}
\begin{equation}\label{eq:L-P}
\mathrm{L}_{P}(w^*) =v^*-(b^*\cdot\nabla)
b^*-{h^*}^{-1}\nabla\big({h^*}^{-1}\big).
\end{equation}
Moreover, we have $L(w^*)=0$ if $(h^*,h^*b^*,h^*d^*,h^*v^*)$ is also a solution to the Darcy MHD \eqref{eq:dd-h}-\eqref{eq:dd-div}.
\end{lemma}

With the above lemma and the nice formula of the relative entropy, we can apply the Gronwall's lemma to estimate the growth of the relative entropy. This is the start point of introducing the concept of dissipative solution to study such degenerate parabolic system.

Now, first, we see that the matrix valued function $Q(w^*)$ in \eqref{eq:dd-rel-1} is a symmetric and its right down $6\times6$ block is always positive definite. Now let us use $I_{n:m}$ to represent the $n\times n$ diagonal matrix whose first $m$ terms are 1 while the rest terms are 0, let $I_{d}$ be the $d\times d$ identity matrix. Then it is easy to verify that for any $\delta\in(0,2)$, there is a constant $r_0=r_0(w^*,\delta,T)$, such that for all $r\geq r_0$ and $(t,x)\in[0,T]\times\mathbb{T}^3$, we have
$$Q(w^*)+rI_{10:4}\geq(2-\delta)I_{10}>0.$$
For the convenience of writing, let us denote,
\begin{equation}
Q_r(w^*)=Q(w^*)+rI_{10:4}.
\end{equation}
Then, \eqref{eq:dd-rel-1} can be written as,
\begin{equation}
\left(\frac{{\rm d}}{{\rm d}t}-r\right)\int\frac{\big|\widetilde{U}\big|^2}{2h} +\int\frac{\widetilde{W}^{\rm T}Q_r(w^*)\widetilde{W}}{2h} +\int \widetilde{W}\cdot\mathrm{L}(w^*)=0.
\end{equation}
We integrate it from 0 to $t$, then we have
\begin{equation}\label{eq:ddin-3}
\int\frac{\big|\widetilde{U}(t)\big|^2}{2h(t)}+\int^{t}_{0}e^{r(t-s)}\left[\int\frac{\widetilde{W}^{{\rm T}}Q_r(w^*)\widetilde{W}}{2h} + \widetilde{W}\cdot\mathrm{L}(w^*)\right]{\rm d}s = e^{rt}\int\frac{\big|\widetilde{U}(0)\big|^2}{2h(0)}.
\end{equation}
Notice that the above equality have a nice structure since the left hand side is in fact a convex functional of $(h,B,D,P)$. It is even possible to extend the meaning of the equality to Borel measures (cf. \cite{DT}). In our case, it is quite simple and direct. For any Borel measure $\rho\in C(\mathbb{T}^3,\mathbb{R})'$ and vector-valued Borel measure $U\in C(\mathbb{T}^3,\mathbb{R}^4)'$, we define
\begin{equation}\label{eq:dddfl}
\Lambda(\rho,U)=\sup\left\{
\int_{\mathbb{T}^3} a\rho+A\cdot U,\;\;\;a+\frac{1}{2}|A|^2\le 0\right\}
\in [0,+\infty],
\end{equation}
where the supremum is taken over all $(a,A)\in C(\mathbb{T}^d;\mathbb{R}\times\mathbb{R}^4)$. As an easy exercise, we can check that
\begin{equation}\label{eq:lam1}
\Lambda(\rho,U)=\begin{cases}
\displaystyle{\frac{1}{2}\int_{\mathbb{T}^3}|u|^2\rho}, & \rho\geq 0,\;U\ll\rho,\;U=u\rho,\;u\in L^2_{\rho}\\
+\infty, & {\rm otherwise}
\end{cases}
\end{equation}
So we can see that $\Lambda(\rho,U)$ is somehow a generalization of the functional $\int\frac{|U|^2}{2\rho}$ to Borel measures. Similarly, we can define a functional in terms of the space time integral of
$$\int_s^t\int_{\mathbb{T}^3}\frac{W^{\rm T}QW}{2\rho}$$
More precisely, for any Borel measure $\rho\in C([0,T]\times\mathbb{T}^3,\mathbb{R})'$, vector-valued Borel measure $W\in C([0,T]\times\mathbb{T}^3,\mathbb{R}^{10})'$, and matrix valued function $Q\in C([0,T]\times\mathbb{T}^3,\mathbb{R}^{10\times10})$ which is always positive definite, we define
\begin{equation}\label{eq:dddft}
\widetilde{\Lambda}(\rho,W,Q;s,t)=\sup\left\{
\int_s^t\int_{\mathbb{T}^3} a\rho+A\cdot W,\;\;\;a+\frac{1}{2}|\sqrt{Q^{-1}}A|^2\leq 0\right\}
\in [0,+\infty],
\end{equation}
where the supremum is taken over all $(a,A)\in C([s,t]\times\mathbb{T}^d;\mathbb{R}\times\mathbb{R}^{10})$, $0\leq s<t\leq T$. Similarly, we have
\begin{equation}\label{eq:lam2}
\widetilde{\Lambda}(\rho,W,Q;s,t)=\begin{cases}
\displaystyle{\frac{1}{2}\int_s^t\int_{\mathbb{T}^3}|\sqrt{Q}w|^2\rho}, & {\rm on\ }[s,t],\;\rho\geq 0,\;W\ll\rho,\;W=w\rho,\;w\in L^2_{\rho}\\
+\infty, & {\rm otherwise}
\end{cases}
\end{equation}
By using the above defined functional, \eqref{eq:ddin-3} can be written as
\begin{equation}
e^{-rt}\Lambda(h(t),\widetilde{U}(t)) + \widetilde{\Lambda}(h,\widetilde{W},e^{-rs}Q_r(w^*);0,t) + R(t) =\Lambda(h(0),\widetilde{U}(0)).
\end{equation}
where
$$R(t)=\int^{t}_{0}\int_{\mathbb{T}^3} e^{-rs}\widetilde{W}\cdot\mathrm{L}(w^*).$$
Now, instead of having an equality, we would like to look for all measure valued solutions such that their relative entropies $\Lambda(h,\widetilde{U})$ are less than the initial data in \eqref{eq:ddin-3}. This is the idea of introducing the concept of dissipative solution.

%%%%%%%%%%%%%%%%%%%%%%%%%%%%%%%%%%%%%%%%%%%%%%%%%%

\subsection{Definition of the dissipative solutions}

With the help of \eqref{eq:ddin-3} and the introducing of $\Lambda(h,U)$. Now we can give a definition of the dissipative solution of (DMHD). Our definition reads,

\begin{definition}\label{def:dddf1}
We say that $(h,B,D,P)$ with $h\in C([0,T],C(\mathbb{T}^3,\mathbb{R})'_{w^{*}})$, $B\in C([0,T],C(\mathbb{T}^3,\mathbb{R}^3)'_{w^{*}})$, $D,P\in C([0,T]\times\mathbb{T}^3,\mathbb{R}^3)'$, is a dissipative solution of (DMHD) \eqref{eq:dd-h}-\eqref{eq:dd-div} with initial data $h_{0}\in C(\mathbb{T}^3,\mathbb{R})',B_{0}\in C(\mathbb{T}^3,\mathbb{R}^{3})'$ if and only if
\\
\\
{\rm\bf(i)} $h(0)=h_{0}$, $B(0)=B_{0}$, $\Lambda(h_{0},U_{0})<\infty$, where $U_0=(\mathcal{L},B_0)$, $\mathcal{L}$ is the Lebesgue measure on $\mathbb{T}^3$.
\\
{\rm\bf(ii)} $(h,B)$ is bounded in $C^{0,\frac{1}{2}}([0,T],C(\mathbb{T}^3,\mathbb{R}^4)'_{w^*})$ by some constant that depends only on $T$ and $(h_0,B_0)$.
\\
{\rm\bf(iii)} \eqref{eq:dd-h} and \eqref{eq:dd-div} is satisfied in the sense of distributions. More precisely, for all $u\in C^1([0,T]\times\mathbb{T}^3,\mathbb{R})$ and $t\in[0,T]$, we have
\begin{equation}\label{eq:dddf-1}
\int_0^t\int_{\mathbb{T}^3}\partial_t u\; h + \nabla u\cdot P =\int_{\mathbb{T}^3}u(t)h(t)-\int_{\mathbb{T}^3} u(0)h(0)
\end{equation}
\begin{equation}\label{eq:dddf-2}
\int_{\mathbb{T}^3} \nabla u(t)\cdot B(t)=0
\end{equation}
\\
{\rm\bf(iv)} For all $t\in[0,T]$ and all $v^{*},b^{*},d^{*}\in C^{1}([0,T]\times\mathbb{T}^3,\mathbb{R}^3),\;0<h^{*}\in C^{1}([0,T]\times\mathbb{T}^3,\mathbb{R})$ and all real number $r\geq r_0$, the following inequality always holds
\begin{equation}\label{eq:dddf-3}
e^{-rt}\Lambda(h(t),\widetilde{U}(t)) + \widetilde{\Lambda}(h,\widetilde{W},e^{-rs}Q_r(w^*);0,t) + R(t) \leq\Lambda(h(0),\widetilde{U}(0))
\end{equation}
where
$$\widetilde{U}=\big(\mathcal{L}-{h^*}^{-1}h,B-hb^{*}\big),\;\;\widetilde{W}=\left(\widetilde{U},D-hd^{*},P-hv^{*}\right)$$
$$w^*=({h^*}^{-1},b^*,d^*,v^*),$$
$Q_r(w^*)$ is a symmetric matrix defined by $Q_r(w^*)=Q(w^*)+rI_{10:4}$, where $Q(w^*)$ is defined in \eqref{eq:ddm-1}. $r_0$ is a constant chosen such that $Q_{r_0}(w^*)\geq I_{10}$ for all $(t,x)\in [0,T]\times\mathbb{T}^3$. $R(t)$ is a functional that depends linearly on $\widetilde{W}$ with the expression
\begin{equation}
R(t)=\int^{t}_{0}\int_{\mathbb{T}^3} e^{-rs}\widetilde{W}\cdot\mathrm{L}(w^*).
\end{equation}
where $\mathrm{L}(w^*)$ is defined in \eqref{eq:L-h}-\eqref{eq:L-P}.

\end{definition}

Note that in the above definition, $C(\mathbb{T}^3,\mathbb{R})'_{w^{*}}$ is the dual space of $C(\mathbb{T}^3,\mathbb{R})$ equipped with the weak-$*$ topology. It is a metrizable space, we can define a metric that is consistent with the weak-$*$ topology, for example, we can take
\begin{equation}\label{eq:ddpf-inapm}
d(\rho,\rho')=\sum_{n\geq0}2^{-n}\frac{\left|\big\langle \rho-\rho',f_{n}\big\rangle\right|}{1+\left|\big\langle \rho-\rho',f_{n}\big\rangle\right|}
\end{equation}
where $\{f_{n}\}_{n\geq0}$ is a smooth dense subset of the separable space $C(\mathbb{T}^3,\mathbb{R})$, $\langle\cdot,\cdot\rangle$ denote the duality pairing of $C(\mathbb{T}^3,\mathbb{R})$ with its dual space.

%%%%%%%%%%%%%%%%%%%%%%%%%%%%%%%%%%%%%%%%%%%%%%%%%%
%%%%%%%%%%%%%%%%%%%%%%%%%%%%%%%%%%%%%%%%%%%%%%%%%%
%%%%%%%%%%%%%%%%%%%%%%%%%%%%%%%%%%%%%%%%%%%%%%%%%%

\section{Properties of the Dissipative Solutions}\label{sec:4}

In this section, we will study some properties of the dissipative solutions that we define in the previous part. We will show that the dissipative solutions satisfy the weak-strong uniqueness, the set of solutions are convex and compact in the weak-$*$ topology, and under what situation, the dissipative solutions become strong solutions.

%%%%%%%%%%%%%%%%%%%%%%%%%%%%%%%%%%%%%%%%%%%%%%%%%%

\subsection{Consistency with smooth solutions}

In this part, let's look at a very interesting question about the dissipative solution. It has been shown that, in \lmref{lm:dd-entropy}, any strong solution $(h,B,D,P)$ to (DMHD) satisfies the energy dissipative inequality \eqref{eq:dddf-3}, so it is naturally a dissipative solution. On the contrary, it is generally not true that a dissipative solution is a strong solution. However, if we know that the dissipative solution has some regularity (for example $C^1$ solutions), then the reverse statement is true. We summarize our result in the following proposition:
\begin{proposition}
Suppose $(h,B,D,P)\in C^1([0,T]\times\mathbb{T}^3,\mathbb{R}^{10})$ is a dissipative solution to (DMHD) \eqref{eq:dd-h}-\eqref{eq:dd-div} in the sense of \defref{def:dddf1}, then is must be a strong solution.
\end{proposition}

\begin{proof}
The proof follows almost the same computation as in \lmref{lm:dd-entropy}. First, the equation \eqref{eq:dd-h} and \eqref{eq:dd-div} is naturally satisfied by the definition of dissipative solution. Our goal is to show that $(h,B,D,P)$ also satisfy \eqref{eq:dd-B},\eqref{eq:dd-con} in the strong sense. o prove this, we denote
$$\phi=\partial_{t}B +\nabla\times\left(\frac{D + B \times P}{h}\right),\;\;\;\psi=D - \nabla\times\left(\frac{B}{h}\right)$$
$$\varphi=P - \nabla\cdot\left(\frac{B\otimes B}{h}\right)-\nabla\left(\frac{1}{h}\right)$$
We only need to prove that $\phi=\psi=\varphi=0$. In fact, for any test function $v^{*},b^{*},d^{*}\in C^{1}([0,T]\times\mathbb{T}^3,\mathbb{R}^3),h^{*}>0\in C^{1}([0,T]\times\mathbb{T}^3,\mathbb{R})$, we follow the same computation as in \lmref{lm:dd-entropy} (this is shown in the Appendix), then we can the following equality,
$$
\frac{{\rm d}}{{\rm d}t}\int\frac{\big|\widetilde{U}\big|^2}{2h} + \int\frac{W^{\rm T}Q(w^*)W}{2h} +\int\widetilde{W}\cdot\mathrm{L}(w^*) = \int\Big[\phi\cdot \big(b-b^*\big) +\psi\cdot \big(d-d^*\big) + \varphi\cdot \big(v-v^*\big) \Big]
$$
Now, let's set $b^*=b-\phi$, $d^*=d-\psi$, $v^*=v-\varphi$, then we have
$$
\frac{{\rm d}}{{\rm d}t}\int\frac{\big|\widetilde{U}\big|^2}{2h} + \int\frac{W^{\rm T}Q(w^*)W}{2h} +\int\widetilde{W}\cdot\mathrm{L}(w^*) = \int\big(\phi^2 +\psi^2 + \varphi^2 \big)
$$
For $r$ big enough, we have
\begin{multline*}
e^{-rT}\int\frac{\big|\widetilde{U}(T)\big|^2}{2h(T)}+\int^{T}_{0}e^{-rs}\left[\int\frac{\widetilde{W}^{{\rm T}}Q_r(w^*)\widetilde{W}}{2h} + \widetilde{W}\cdot\mathrm{L}(w^*)\right]{\rm d}s -\int\frac{\big|\widetilde{U}(0)\big|^2}{2h(0)} \\ =\int^{T}_{0}\int e^{-rs}\big(\phi^2 +\psi^2 + \varphi^2 \big)
\end{multline*}
By the definition of dissipative solution, we have
$$
\int^{T}_{0}\int e^{-rs}\big(\phi^2 +\psi^2 + \varphi^2 \big)\leq 0
$$
This implies that $\phi\equiv\varphi\equiv\psi\equiv0$, which completes the proof.

\end{proof}

%%%%%%%%%%%%%%%%%%%%%%%%%%%%%%%%%%%%%%%%%%%%%%%%%%

\subsection{Weak-strong uniqueness and stability result}

The weak-strong uniqueness is essentially an important property for a suitable concept of ``weak'' solution of a given evolution system. By the weak-strong uniqueness, we mean that any weak solution must coincide with a strong solution emanating from the same initial data as long as the latter exists. In other words, the strong solutions must be unique within the class of weak solutions. This kind of problem has been widely studied in various kinds of equations (Navier-Stokes, Euler, etc.), even for measure valued solutions \cite{BDS}. In our (DMHD), we will show that the dissipative solution also enjoy this kind of property. First, let's us show a stability estimate.

\begin{proposition}\label{prop:dd-stability}
Suppose that $(h^*>0,B^*,D^*,P^*)$ is a classical (at least $C^1$) solution of (DMHD) \eqref{eq:dd-h}-\eqref{eq:dd-div} with initial value $(h^{*},B^{*})|_{t=0}=(h_{0}^{*},B_{0}^{*})$. $(h,B,D,P)$ is a dissipative solution with initial value $(h,B)|_{t=0}=(h_{0},B_{0})$. Let us denote
$$\widetilde{U}=\big(\mathcal{L}-h{h^{*}}^{-1},B-hb^{*}\big),\;\;\;\widetilde{W}=\big(\widetilde{U},D-hd^{*},P-hv^{*}\big)$$
where $b^{*}=B^*/{h^{*}},\;d^{*}=D^*/{h^{*}},\;v^*=P^*/{h^*}$. Then, for any $t\in[0,T]$, there exist a constant $C$ that depends only on the choice of $(h^{*},B^{*},D^{*},P^{*})$, the value of $\Lambda(h_0,\widetilde{U}_0)$ and $T$, such that the following estimates hold
\begin{equation}\label{eq:dd-sta-est1}
\|\widetilde{U}(t)\|_{TV}^2\leq Ce^{Ct}\Lambda(h_0,\widetilde{U}_0),\;\;\;\|\widetilde{W}\|_{TV^{*}}^{2}\leq Ce^{CT}\Lambda(h_0,\widetilde{U}_0).
\end{equation}
Here $\|\cdot\|_{TV}$, $\|\cdot\|_{TV^{*}}$ respectively represent the total variation of measures on $\mathbb{T}^3$ and $[0,T]\times\mathbb{T}^3$. Furthermore, we have that
\begin{equation}\label{eq:dd-sta-est2}
\|h(t)-h^{*}(t)\|_{TV}^2,\;\;\|B(t)-B^{*}(t)\|_{TV}^2\leq Ce^{Ct}\Lambda(h_0,\widetilde{U}_0),
\end{equation}
\begin{equation}\label{eq:dd-sta-est3}
\|D-D^{*}\|_{TV^{*}}^{2},\;\;\|P-P^{*}\|_{TV^{*}}^2\leq Ce^{CT}\Lambda(h_0,\widetilde{U}_0).
\end{equation}

\end{proposition}

\begin{proof}
The proof is very simple. We just need to take $(h^*,b^*,d^*,v^*)$ defined in the proposition as our test functions and apply it to the energy dissipative inequality \eqref{eq:dddf-3}. Because $(h^{*},B^{*},D^{*},P^{*})$ is a strong solution, so we have $\mathrm{L}(w^*)\equiv0$, where $w^*=({h^*}^{-1},b^*,d^*,v^*)$. Let $r_0>0$ be a constant such that $Q_{r_0}(w^*)\geq I_{10}$ for all $(t,x)\in [0,T]\times\mathbb{T}^3$. Then for $r\geq r_0$ and $t\in[0,T]$, \eqref{eq:dddf-3} gives
$$e^{-rt}\Lambda(h(t),\widetilde{U}(t)) + \widetilde{\Lambda}(h,\widetilde{W},e^{-rs}Q_r(w^*);0,t) \leq\Lambda(h_0,\widetilde{U}_0).$$
So we have
$$\Lambda(h(t),\widetilde{U}(t))\leq e^{rt}\Lambda(h_0,\widetilde{U}_0),$$
and, since $e^{-rs}Q_r(w^*)\geq e^{-rT}I_{10}$, we have
$$e^{-rT}\widetilde{\Lambda}(h,\widetilde{W},I_{10};0,T) \leq \widetilde{\Lambda}(h,\widetilde{W},e^{-rs}Q_r(w^*);0,T) \leq\Lambda(h_0,\widetilde{U}_0).$$
Now, since $h$ satisfies \eqref{eq:dd-h} in the sense of distributions, we have
$$\int_{\mathbb{T}^3}h(t)=\int_{\mathbb{T}^3}h_0.$$
Then, by the expression of $\Lambda$, $\widetilde{\Lambda}$ in \eqref{eq:lam1},\eqref{eq:lam2}, and Cauchy-Schwarz inequality, we have
$$\|\widetilde{U}(t)\|_{TV}^2\leq2\Lambda(h(t),\widetilde{U}(t))\int_{\mathbb{T}^3}h(t)\leq 2e^{rt}\Lambda(h_{0},\widetilde{U}_{0})\int_{\mathbb{T}^3}h_0$$
$$\|\widetilde{W}\|_{TV^{*}}^{2}
\leq 2\widetilde{\Lambda}(h,\widetilde{W},I_{10};0,T)\int_0^T\int_{\mathbb{T}^3}h\leq 2Te^{rT}\Lambda(h_{0},\widetilde{U}_{0})\int_{\mathbb{T}^3}h_0$$
Now, we would like to estimate the value of $\|h_0\|_{TV}$, given the value of $\Lambda_{0}=\Lambda(h_{0},\widetilde{U}_{0})$ and $h^*_0$. Since
$$2\Lambda_{0}\int_{\mathbb{T}^3}h_0\geq \|\widetilde{U}_{0}\|_{TV}^2 \geq  \left\|\mathcal{L}-h_{0}{h_{0}^{*}}^{-1}\right\|_{TV}^2\geq \|h_{0}^{*}\|_{\infty}^{-2} \left(\int_{\mathbb{T}^3} |h_{0}^{*}-h_{0}|\right)^2$$
So we have
$$
\int_{\mathbb{T}^3} |h_{0}^{*}-h_{0}|\leq \|h_{0}^{*}\|_{\infty} \sqrt{ 2\Lambda_{0}\int_{\mathbb{T}^3}h_0}
$$
Then we have that
$$
\int_{\mathbb{T}^3}h_0\leq \int_{\mathbb{T}^3} |h_{0}^{*}-h_{0}| +\int_{\mathbb{T}^3} h_{0}^{*} \leq \|h_{0}^{*}\|_{\infty} \sqrt{ 2\Lambda_{0}\int_{\mathbb{T}^3}h_0} + \int_{\mathbb{T}^3} h_{0}^{*}
$$
$$
\leq \frac{1}{2}\int_{\mathbb{T}^3}h_0 + \|h_{0}^{*}\|_{\infty}^2\Lambda_{0} + \int_{\mathbb{T}^3} h_{0}^{*}
$$
So we have the estimate
$$
\int_{\mathbb{T}^3}h_0\leq 2 \left(\|h_{0}^{*}\|_{\infty}^2\Lambda_{0} + \int_{\mathbb{T}^3} h_{0}^{*}\right)
$$
Combining the above results, we can find a constant $C$ which depends only on $(h^{*},B^{*},D^{*},P^{*})$, $\Lambda(h_0,\widetilde{U}_0)$ and $T$, such that the estimate \eqref{eq:dd-sta-est1} is satisfied. \eqref{eq:dd-sta-est2},\eqref{eq:dd-sta-est3} are also easy to prove, since we have
\begin{equation*}
\begin{array}{r@{}l}
\|B(t)-B^{*}(t)\|_{TV} & \displaystyle{\leq \|B(t)-h(t)b^{*}(t)\|_{TV} + \|B^{*}(t)-h(t)b^{*}(t)\|_{TV}}\\
& \displaystyle{ \leq \|B(t)-h(t)b^{*}(t)\|_{TV} + \|B^{*}(t)\|_{\infty}\left\|\mathcal{L}-h(t){h^*}^{-1}(t)\right\|_{TV} }\\
& \displaystyle{ \leq (1+\|B^{*}\|_{\infty})\|\widetilde{U}(t)\|_{TV}
}
\end{array}
\end{equation*}
Similarly, we have
\begin{equation*}
\begin{array}{r@{}l}
\|D-D^{*}\|_{TV^{*}} & \displaystyle{ \leq \|D-hd^{*}\|_{TV^{*}} + \|D^{*}\|_{\infty}\left\|\mathcal{L}-h{h^{*}}^{-1}\right\|_{TV^{*}} }\\
& \displaystyle{ \leq (1+\|D^{*}\|_{\infty})\|\widetilde{W}\|_{TV^{*},}
}
\end{array}
\end{equation*}
$$\|h(t)-h^{*}(t)\|_{TV}\leq \|h^{*}\|_{\infty}\|\widetilde{U}(t)\|_{TV},\;\;\|P-P^{*}\|_{TV^{*}}\leq (1+\|P^{*}\|_{\infty})\|\widetilde{W}\|_{TV^{*}}$$
Then by \eqref{eq:dd-sta-est1}, we can quickly get our desired result.

\end{proof}

The above proposition gives us an estimate of the distance of two different dissipative solution as time evolves. As a direct consequence, we immediately get the weak-strong uniqueness for the dissipative solutions.

\begin{proposition}\label{prop:dd-weak-strong}
Suppose that $(h>0,B,D,P)$ is a classical (at least $C^1$) solution to (DMHD) \eqref{eq:dd-h}-\eqref{eq:dd-div} with initial value $(h,B)|_{t=0}=(h_{0},B_{0})$, then it is the unique dissipative solution to (DMHD) with the same initial value.

\end{proposition}

\begin{proof}
The proof is very simple. Suppose there is another dissipative solution $(h',B',D',P')$. Since $\widetilde{U}_0=0$, the estimates \eqref{eq:dd-sta-est2},\eqref{eq:dd-sta-est3} in the previous proposition just give us the uniqueness.

\end{proof}

%%%%%%%%%%%%%%%%%%%%%%%%%%%%%%%%%%%%%%%%%%%%

\subsection{Weak compactness}

From the previous parts, we know that, if we have a smooth dissipative solution, then it should be a strong solution, and, therefore, it should be the unique dissipative solution with the same initial data. However, in general, dissipative solutions are not usually that regular. A natural question is that, what happens to the dissipative solutions that are not smooth? Are they unique? If not, what can we conclude for the set of dissipative solutions? In this part, we will show that the dissipative solutions satisfy the ``weak compactness'' property (i.e. any sequence of dissipative solutions has accumulations
points, in a suitable weak sense, and each of them is still a dissipative solution). We summarize our result in the following theorem:

\begin{theorem}\label{thm:dd-weak-comp}
For any initial data $B_{0}\in C(\mathbb{T}^3,\mathbb{R}^3)',\;h_{0}\in C(\mathbb{T}^3,\mathbb{R})'$, satisfying that $\nabla\cdot B_{0}=0$ in the sense of distributions and $\Lambda(h_{0},U_{0})<\infty$, let $\mathcal{A}$ be the set of all dissipative solutions $(h,B,D,P)$ to (DMHD) \eqref{eq:dd-h}-\eqref{eq:dd-div} with initial data $(h_0,B_0)$. Then $\mathcal{A}$ is a non-empty convex compact set in the space $C([0,T],C(\mathbb{T}^3,\mathbb{R}\times\mathbb{R}^3)'_{w^*})\times C([0,T]\times\mathbb{T}^3,\mathbb{R}^3\times\mathbb{R}^3)'_{w^*}$.
\end{theorem}

\begin{proof}%[Proof of \thmref{thm:dd-weak-comp}]

The non-emptiness of $\mathcal{A}$ refers just to the existence of the dissipative solutions. The proof is a little lengthy, we leave the existence proof in \secref{sec:6} and \secref{sec:7}. Here, let's prove the convexity and compactness. The convexity of $\mathcal{A}$ is quite easy. As we can see in \defref{def:dddf1}, \eqref{eq:dddf-1},\eqref{eq:dddf-2} are linear equations. So it is always satisfied under any convex combination of dissipative solutions. Since the functional $\Lambda,\widetilde{\Lambda}$ are convex, so \eqref{eq:dddf-3} is also satisfied. So we know that the set $\mathcal{A}$ is convex. Now let's show the compactness. Since $h^*\equiv 1$, $B^*=D^*=P^*=0$ is a trivial solution, then for any family of dissipative solutions $(h_n,B_n,D_n,P_n)$ with initial data $(h_0,B_0)$, by \propref{prop:dd-stability}, there exist a constant $C$ that depends only on $(h_0,B_0)$ and $T$, such that
$$\|h_n(t)-1\|_{TV},\;\;\|B_n(t)\|_{TV},\;\;\|D_n\|_{TV^{*}},\;\;\|P_n\|_{TV^{*}}\leq C.$$
Since $(h_n,B_n)$ are uniformly bounded in $C^{0,\frac{1}{2}}([0,T],C(\mathbb{T}^3,\mathbb{R}^4)'_{w^*})$, then up to a subsequence, $(h_n,B_n)$ converge to some function $(h,B)$ in $C([0,T],C(\mathbb{T}^3,\mathbb{R}^4)'_{w^*})$. Also, since $(D_n,P_n)$ are uniformly bounded in $C([0,T]\times\mathbb{T}^3,\mathbb{R}^3\times\mathbb{R}^3)'$, then up to a subsequence, $(D_n,P_n)$ converge weakly-$*$ to some $(D,P)$ in $C([0,T]\times\mathbb{T}^3,\mathbb{R}^3\times\mathbb{R}^3)'$. Now we only need to prove that $(h,B,D,P)$ is a dissipative solution. This is easy since  \eqref{eq:dddf-1},\eqref{eq:dddf-2} and \eqref{eq:dddf-3} are weakly stable.

\end{proof}

%%%%%%%%%%%%%%%%%%%%%%%%%%%%%%%%%%%%%%%%%%%%%
%%%%%%%%%%%%%%%%%%%%%%%%%%%%%%%%%%%%%%%%%%%%%
%%%%%%%%%%%%%%%%%%%%%%%%%%%%%%%%%%%%%%%%%%%%%

\section{Comparison with smooth solutions of the ABI equations}\label{sec:5}

As it has been shown in \secref{sec:2}, we can get our (DMHD) out of the augmented BI equations \eqref{eq:ABI-1},\eqref{eq:ABI-2} through the quadratic change of the time variable $\theta\rightarrow t^2/2$. Now a natural question is that, since the (DMHD) can be seen as an approximation of the ABI equations, how about the solutions of these two systems of equations? Are they close to each other when the initial data are the same? With our concept of the dissipative solution, it is possible to give an answer.

In \cite{Br-W}, it is shown that ABI equations can be rewritten as a symmetric hyperbolic system of conservation laws. So smooth solutions exist at least in a short period of time for smooth initial data, see \cite{Dafer}. Now, for any smooth function $h_0>0$, $B_0$, $\nabla\cdot B_0=0$, there exist a time interval $[0,t_{0}]$, such that there exists a smooth solution $(h',B',D',P')$ to the augmented BI system with initial value $(h_{0},B_{0},0,0)$. We will compare the smooth solution $(h',B',D',P')$ with the dissipative solution $(h,B,D,P)$ to the Darcy MHD on $[0,T]$, $T\geq t_{0}^2/2$, with the same initial value $(h_0,B_0)$. Our estimates are in the following proposition.

\begin{proposition}
Suppose $(h',B',D',P')$ is a smooth solution to the augmented BI equations \eqref{eq:ABI-1},\eqref{eq:ABI-2} on $[0,t_{0}]$ with smooth initial data $(h_{0},B_{0},0,0)$. $(h,B,D,P)$ is a dissipative solution to (DMHD) \eqref{eq:dd-h}-\eqref{eq:dd-div} on $[0,T]$, $T\geq t_{0}^2/2$, with the same initial data $(h_{0},B_{0})$. Then there exists a constant $C$ that depends only on $(h',B',D',P')$ and $t_0$, such that for any $t\in[0,t_0]$, we have
\begin{equation}\label{eq:dd-abi-1}
\|h'(t)-h(t^2/2)\|_{TV},\;\;\;\|B'(t)-B(t^2/2)\|_{TV}\leq C t^3
\end{equation}
\begin{equation}\label{eq:dd-abi-2}
\big|D'(s)-sD(s^2/2)\big|([0,t]\times\mathbb{T}^3),\;\;\;\big|P'(s)-sP(s^2/2)\big|([0,t]\times\mathbb{T}^3)\leq C t^4
\end{equation}
Here $|\cdot|$ represents the variation of the vector-valued measures. $sD(s^2/2)$, $sP(s^2/2)$ denote the vector-valued Borel measures on $[0,t_0]\times\mathbb{T}^3$ defined in the way such that, for all $\varphi\in C([0,t_0]\times\mathbb{T}^3,\mathbb{R}^3)$, we have
$$\int_0^{t_0}\int_{\mathbb{T}^3} \varphi(s)\cdot sD(s^2/2)= \int_0^{t_0^2/2}\int_{\mathbb{T}^3}\varphi(\sqrt{2s})\cdot D(s)$$
$$\int_0^{t_0}\int_{\mathbb{T}^3} \varphi(s)\cdot sP(s^2/2)= \int_0^{t_0^2/2}\int_{\mathbb{T}^3}\varphi(\sqrt{2s})\cdot P(s)$$
\end{proposition}

\begin{proof}
First, since $(h',B',D',P')$ is a smooth solution to the augmented BI equations, then the non-conservative variables $$(\tau',b',d',v')=(1/h',B'/h',D'/h',P'/h')$$
should satisfy the following equations
$$
\partial_t b' +(v'\cdot\nabla)b'=(b'\cdot\nabla)v'-\tau'\nabla\times d',\;\;\partial_t d' +(v'\cdot\nabla)d'=(d'\cdot\nabla)v'+\tau'\nabla\times b',
$$
$$
\partial_t \tau' +(v'\cdot\nabla)\tau'=\tau'\nabla\cdot v',\;\;\partial_t v' +(v'\cdot\nabla)v'=(b'\cdot\nabla)b'+(d'\cdot\nabla)d'+\tau'\nabla\tau',
$$
Now let's take our test function $h^*,b^*,d^*,v^*$ defined as following
\begin{equation}\label{eq:cp-1}
\begin{array}{ll}
\displaystyle{h^{*}(\theta,x)=h'(\sqrt{2\theta},x)} & \displaystyle{b^{*}(\theta,x)=b'(\sqrt{2\theta},x)}\\
\displaystyle{d^{*}(\theta,x)=\frac{d'(\sqrt{2\theta},x)}{\sqrt{2\theta}}} & \displaystyle{v^{*}(\theta,x)=\frac{v'(\sqrt{2\theta},x)}{\sqrt{2\theta}}}
\end{array}
\end{equation}
We should notice that $d^*,v^*$ is well defined and continuous with value $\partial_{t}d'(0),\partial_{t}v'(0)$ at time $\theta=0$. Moreover, it is easy to verify that $\partial_{t}^{2}d'(0)=\partial_{t}^{2}v'(0)=0$, so we know that $h^{*},b^{*},d^{*},v^{*}$ are $C^{1}$ functions with $\partial_{\theta}d^{*}(0)=\frac{1}{3}\partial_{t}^{3}d'(0),\partial_{\theta}v^{*}(0)=\frac{1}{3}\partial_{t}^{3}v'(0)$.
Now, let's do the change of time $\theta=t^2/2$, \eqref{eq:cp-1} means
$$
\tau'(t,x)={h^{*}}^{-1}(\theta,x),\;b'(t,x)={b^{*}}(\theta,x),\;d'(t,x)=td^*(\theta,x),\;v'(t,x)=tv^*(\theta,x)
$$
Then our test function should satisfy the following equations,
$$
\partial_\theta ({h^*}^{-1}) +(v^*\cdot\nabla){h^*}^{-1}={h^*}^{-1}\nabla\cdot v^*
$$
$$
\partial_\theta b^* +(v^*\cdot\nabla)b^*=(b^*\cdot\nabla)v^*-{h^*}^{-1}\nabla\times d^*
$$
$$
d^*-{h^*}^{-1}\nabla\times b^*=-2\theta\big[\partial_\theta d^*+(v^*\cdot\nabla)d^*-(d^*\cdot\nabla)v^*\big],
$$
$$
v^*- (b^*\cdot\nabla)b^*-{h^*}^{-1}\nabla({h^*}^{-1}) =-2\theta\big[\partial_\theta v^*+(v^*\cdot\nabla)v^* -(d^*\cdot\nabla)d^*\big],
$$
So we have that
$$\mathrm{L}_{h}(w^*)=\mathrm{L}_{B}(w^*)=0,\;\;\mathrm{L}_{D}(w^*)=-2\theta \psi_d^*,\;\;\mathrm{L}_{P}(w^*)=-2\theta \psi_v^*,$$
where $\psi_d^*,\psi_v^*$ are continuous functions with the following expressions
$$\psi_d^*=\partial_\theta d^*+(v^*\cdot\nabla)d^*-(d^*\cdot\nabla)v^*,$$
$$\psi_v^*=\partial_\theta v^*+(v^*\cdot\nabla)v^* -(d^*\cdot\nabla)d^*.$$
Now, for the dissipative solution $(h,B,D,P)$ to (DMHD), we denote as usual,
$$\widetilde{U}=\big(\mathcal{L}-h{h^*}^{-1},B-hb^*\big),\;\;\; \widetilde{W}=\big(\widetilde{U},D-hd^*,P-hv^*\big)$$
Since the initial value are the same, we have $\widetilde{U}(0)=0$. Now we follow the definition of dissipative solution, there exists a constant $r>0$ such that $Q_{r}(w^*)\geq I_{10}$ for all $(\theta,x)\in [0,t^2_0/2]\times\mathbb{T}^3$. By \eqref{eq:dddf-3}, we have, for $\theta\in[0,t^2_0/2]$,
\begin{equation}\label{eq:comp1}
e^{-r\theta}\Lambda(h(\theta),\widetilde{U}(\theta)) + \widetilde{\Lambda}(h,\widetilde{W},e^{-r\theta'}Q_r(w^*);0,\theta) + R(\theta)\leq 0.
\end{equation}
where
$$R(\theta)= -2\theta\int_0^\theta\int_{\mathbb{T}^3}\Big[\psi_d^*\cdot(D-hd^*) + \psi_v^*\cdot(P-hv^*)\Big].$$
By Cauchy-Schwarz inequality, we have
$$\|\widetilde{U}(\theta)\|_{TV}^2\leq2\Lambda(h(\theta),\widetilde{U}(\theta))\int_{\mathbb{T}^3}h(\theta)=2\Lambda(h(\theta),\widetilde{U}(\theta))\int_{\mathbb{T}^3}h_0$$
$$\Big(\big|\widetilde{W}\big|([0,\theta]\times\mathbb{T}^3)\Big)^{2}\leq
2\theta e^{r\theta}\widetilde{\Lambda}(h,\widetilde{W},e^{-r\theta'}Q_r(w^*);0,\theta)\int_{\mathbb{T}^3}h_0$$
Now let $M=\left\|\psi_d^*\right\|_{\infty}+\left\|\psi_v^*\right\|_{\infty}$, then we have
$$
|R(\theta)| \leq 2\theta \big|\widetilde{W}\big|([0,\theta]\times\mathbb{T}^3)M
$$
Therefore, \eqref{eq:comp1} implies,
\begin{equation}
\theta\|\widetilde{U}(\theta)\|_{TV}^2 + \Big(\big|\widetilde{W}\big|([0,\theta]\times\mathbb{T}^3)\Big)^{2}\leq C\theta^2\big|\widetilde{W}\big|([0,\theta]\times\mathbb{T}^3)
\end{equation}
where $C$ is a constant that depends only on $M$, $h_0$, $r$ and $t_0$. Then we have that
$$\big|\widetilde{W}\big|([0,\theta]\times\mathbb{T}^3)\leq C\theta^2,\;\;\;\|\widetilde{U}(\theta)\|_{TV}^2\leq C^2\theta^3.$$
Since
$$
\|h(t^2/2)-h'(t)\|_{TV}=\|h(\theta)-h^*(\theta)\|_{TV}\leq \|h^*\|_{\infty}\|\mathcal{L}-h{h^*}^{-1}\|_{TV}
$$
$$
\leq \|h^*\|_{\infty}\|\widetilde{U}(\theta)\|_{TV} \leq 2^{-\frac{3}{2}}\|h^*\|_{\infty}C t^3
$$
$$
\|B(t^2/2)-B'(t)\|_{TV}\leq \|B(\theta)-h(\theta)b^{*}(\theta)\|_{TV} + \|b^*(\theta)(h(\theta)-h^*(\theta))\|_{TV}
$$
$$
\leq (1+\|h^*b^*\|_{\infty})\|\widetilde{U}(\theta)\|_{TV}\leq 2^{-\frac{3}{2}}(1+\|h^*b^*\|_{\infty})C t^3
$$
so we get \eqref{eq:dd-abi-1}. Now, for $sD(s^2/2)$, we have
$$\big|D'(s)-sD(s^2/2)\big|([0,t]\times\mathbb{T}^3)=\big|h^*(\theta)d^*(\theta)-D(\theta)\big|([0,\theta]\times\mathbb{T}^3)$$
$$\leq (1+\|h^*d^*\|_{\infty})\big|\widetilde{W}\big|([0,\theta]\times\mathbb{T}^3)\leq 2^{-2}(1+\|h^*d^*\|_{\infty})C t^4$$
similarly,
$$\big|P'(s)-sP(s^2/2)\big|([0,t]\times\mathbb{T}^3)\leq 2^{-2}(1+\|h^*v^*\|_{\infty})C t^4$$
so we get \eqref{eq:dd-abi-2}.

\end{proof}

%%%%%%%%%%%%%%%%%%%%%%%%%%%%%%%%%%%%%%%%
%%%%%%%%%%%%%%%%%%%%%%%%%%%%%%%%%%%%%%%%
%%%%%%%%%%%%%%%%%%%%%%%%%%%%%%%%%%%%%%%%

\section{Faedo-Galerkin approximation}\label{sec:6}

In the following two sections, we will mainly focus on the existence theory of the dissipative solutions. In this section, we consider an approximate system of (DMHD). We want to get a dissipative solution of (DMHD) by the approaching of solutions of the approximate system. In fact, we don't really need to solve the approximate system, we only need to find a sequence of approximate solutions on some finite dimensional spaces, which is quite similar to the Faedo-Galerkin method of E. Feireisl \cite{f1,f2,f3}. We consider the following approximate equations
\begin{equation}\label{eq:ddap-1}
\partial_{t}h +\nabla\cdot (hv)=0,
\end{equation}
\begin{equation}\label{eq:ddap-2}
\partial_{t}B + \nabla\times\left(B\times v + d\right)=0,\;\;\;\nabla\cdot B=0
\end{equation}
\begin{equation}\label{eq:ddap-4}
\varepsilon\Big[\partial_{t}(hd) + \nabla\cdot[h(d\otimes v- v\otimes d)]+(-\triangle)^ld\Big] + hd = \nabla\times b,
\end{equation}
\begin{equation}\label{eq:ddap-3}
\varepsilon\Big[\partial_{t}(hv) + \nabla\cdot (hv\otimes v)- (h d\cdot\nabla) d + (-\triangle)^lv\Big]  + hv = \nabla\cdot \left(\frac{B\otimes B}{h}\right) + \nabla\left(h^{-1}\right),
\end{equation}

Here, $0<\varepsilon<1$, we choose $l$ sufficiently big ($l\geq8$). The idea of using these equations as approximate system comes from the way we get the Darcy MHD from augmented BI. The time derivatives of $d,v$ here can ensure that the approximate solutions are continuous with respect to time. We introduce the high order derivatives here to get some regularities that will be useful in showing the existence. Very similar to the case of augmented BI, we have the following formula (the proof is quite straightforward, we leave it to interested readers)
\begin{equation}\label{eq:ddapin-1}
\frac{{\rm d}}{{\rm d}t}\int_{\mathbb{T}^3}\left[\frac{1+B^2}{2h} + \varepsilon \frac{h(v^2+d^2)}{2}\right] + \int_{\mathbb{T}^3} h(v^2+d^2) +\varepsilon\int_{\mathbb{T}^3}\big|\nabla^lv\big|^2 + \big|\nabla^ld\big|^2= 0
\end{equation}

%%%%%%%%%%%%%%%%%%%%%%%%%%%%%%%%%%%%%%%%%%

\subsection{Classical Solution for Fixed $(d,v)$}

Now let us consider the solution of \eqref{eq:ddap-1} and \eqref{eq:ddap-2} when $d,v$ are given smooth functions.

\begin{lemma}\label{lm:ddap1}
Suppose $h_0,B_0$ are smooth functions with $\nabla\cdot B_0=0,\;h_0>0$. Then for any integer $k\geq 1$ and given $d,v\in C([0,T],C^{k+1}(\mathbb{T}^3,\mathbb{R}^3))$, there exists a unique solution $h\in C^1([0,T],C^{k}(\mathbb{T}^3,\mathbb{R}^3))
,B\in C^1([0,T],C^k(\mathbb{T}^3,\mathbb{R}^3))$ to \eqref{eq:ddap-1} and \eqref{eq:ddap-2} with initial data $h_0,B_0$.
\end{lemma}

\begin{proof}
We use the method of characteristics to show the existence of classical solutions. For \eqref{eq:ddap-1}, the solution can be written explicitly as
\begin{equation}\label{eq:ddap-slh}
h(t,x)=h_0\big(\Phi(0,t,x)\big)\exp\left\{-\int^t_0\nabla\cdot v\big(s,\Phi(s,t,x)\big){\rm d}s\right\}
\end{equation}
where $\Phi(t,s,x)\in C^1([0,T]\times[0,T]\times\mathbb{T}^3)$ is the unique solution of
\begin{equation}\label{eq:ddap-slh-1}
\begin{cases}
\partial_{t}\Phi(t,s,x)=v\big(t,\Phi(t,s,x)\big) & 0\leq t\leq T \\ \Phi(s,s,x)=x & 0\leq s\leq T,\;x\in\mathbb{T}^3
\end{cases}
\end{equation}

Because $v\in C([0,T],C^{k+1}(\mathbb{T}^3,\mathbb{R}^3))$, the existence and uniqueness of such solution is obtained directly by Cauchy-Lipschitz Theorem. Moreover, we can proof that the solution $\Phi(t,s,x)\in C^1([0,T],C^1([0,T],C^{k}(\mathbb{T}^3,\mathbb{R}^3)))$.(Take space derivatives on both side of the equation, the new equation is composed of lower derivatives and is linear for the highest older derivatives. By induction, the Cauchy-Lipschitz Theorem gives a solution and by uniqueness, we can show the solution of \eqref{eq:ddap-slh-1} is that sufficiently differentiable.)

%We can proof this conclusion by applying Banach fixed point theorem in the space $C([0,T],C^1(\times[0,T],C^{k+1}(\mathbb{T}^3,\mathbb{R}^3)))$ to find the fixed point of

%\begin{equation}\label{eq:ddap-slh-2}
%\Phi(t,s,x)= x + \int^t_s v\big(\tau,\Phi(\tau,s,x)\big){\rm d}\tau
%\end{equation}

For \eqref{eq:ddap-2}, we also have an explicit expression of the solution
\begin{equation}\label{eq:ddap-slb}
B(t,x)=G\big(t,\Phi(0,t,x)\big)\exp\left\{-\int^t_0\nabla\cdot v\big(s,\Phi(s,t,x)\big){\rm d}s\right\}
\end{equation}
where $G(t,x)\in C^1([0,T]\times\mathbb{T}^3)$ is the unique solution of
\begin{equation}\label{eq:ddap-slb-1}
\begin{cases}
\partial_{t}G(t,x)=\nabla v\big(t,\Phi(t,0,x)\big)\cdot G(t,x) & 0\leq t\leq T \\\qquad\qquad- (\nabla\times d)\big(t,\Phi(t,0,x)\big)\exp\left\{\int^t_0\nabla\cdot v\big(s,\Phi(s,0,x)\big){\rm d}s\right\} & \\ G(0,x)=B_0(x) & x\in\mathbb{T}^3
\end{cases}
\end{equation}

By the same reason, $G(t,x)\in C([0,T],C^{k}(\mathbb{T}^3,\mathbb{R}^3))$. It is not hard verify that $(h,B)$ defined in \eqref{eq:ddap-slh},\eqref{eq:ddap-slb} is indeed a solution. At last, let us look into the uniqueness of solutions. Because the equations are linear with respect to $h$ and $B$, it is easy to show the uniqueness by the $L^2$ estimates of the difference of two solutions. We can also see it from the following lemma.
\end{proof}

Now for any fixed $z=(d,v)\in C([0,T],C^{k+1}(\mathbb{T}^3,\mathbb{R}^6))$, let us denote $h[z,h_0]$, $B[z,B_0]$ the unique solution of \eqref{eq:ddap-1},\eqref{eq:ddap-2} with initial value $h_0,B_0$ at time $t=0$. Then we have the following lemma.

\begin{lemma}\label{lm:ddap2}
Suppose $z=(d,v)\in C([0,T],C^{k+1}(\mathbb{T}^3,\mathbb{R}^6))$, $k\geq1$, $h_0,B_0$ are smooth functions, $\nabla\cdot B_0=0,\;h_0>0$. Then we have\\
{\rm\bf(i)} For all $t\in[0,T],\;x\in\mathbb{T}^3$,
\begin{equation}\label{eq:ddap-e1}
0<e^{-\int_0^t \|\nabla\cdot v(s)\|_{\infty} {\rm d}s}\inf_{x\in\mathbb{T}^3} h_0 \leq h(t,x)\leq e^{\int_0^t \|\nabla\cdot v(s)\|_{\infty} {\rm d}s}\sup_{x\in\mathbb{T}^3} h_0
\end{equation}
{\rm\bf(ii)} Suppose $\displaystyle{\sup_{t\in[0,T]}\|v(t)\|_{C^{k+1}},\;\sup_{t\in[0,T]}\|d(t)\|_{C^{k+1}}\leq M_{k+1}}$, then
\begin{equation}\label{eq:ddap-e2}
\sup_{t\in[0,T]}\big\|h[z,h_0]\big\|_{H^k(\mathbb{T}^3)}\leq C_{k}(T,\|h_0\|_{H^k},M_{k+1})
\end{equation}
\begin{equation}\label{eq:ddap-e3}
\sup_{t\in[0,T]}\big\|B[z,B_0]\big\|_{H^k(\mathbb{T}^3)}\leq C_{k}(T,\|B_0\|_{H^k},M_{k+1})
\end{equation}
{\rm\bf(iii)} For any $\kappa>0$, any $z,\tilde{z}$ belonging to the set
$$\mathcal{C}_{\kappa}=\big\{z\in C([0,T],C^3(\mathbb{T}^3,\mathbb{R}^6))\big| \sup_{t\in[0,T]}\|v(t)\|_{C^{3}},\;\sup_{t\in[0,T]}\|d(t)\|_{C^{3}}\leq \kappa \big\}$$
we have
\begin{equation}\label{eq:ddap-e4}
\sup_{t\in[0,T]}\big\|h[z,h_0]-h[\tilde{z},h_0]\big\|_{H^1(\mathbb{T}^3)}\leq c(T,\|h_0\|_{H^2},\kappa)\sup_{t\in[0,T]}\big\|z(t)-\tilde{z}(t)\big\|_{C^2}
\end{equation}
\begin{equation}\label{eq:ddap-e5}
\sup_{t\in[0,T]}\big\|B[z,B_0]-B[\tilde{z},B_0]\big\|_{H^1(\mathbb{T}^3)}\leq c(T,\|B_0\|_{H^2},\kappa)\sup_{t\in[0,T]}\big\|z(t)-\tilde{z}(t)\big\|_{C^2}
\end{equation}

\end{lemma}

\begin{proof}
(i) This is a direct conclusion from the explicit expression of $h$ in \eqref{eq:ddap-slh}.

(ii) Let's prove \eqref{eq:ddap-e3}. The case for \eqref{eq:ddap-e2} is simpler. For any $\alpha\in\mathbb{N}^3$, $|\alpha|\leq k$, we have
\begin{multline*}
\partial_{t}(\partial^{\alpha}B_i) +\partial_jv_j\partial^{\alpha}B_i + v_j\partial_j(\partial^{\alpha}B_i)-\partial_jv_i\partial^{\alpha}B_j + \epsilon_{ijk}\partial^{\alpha}(\partial_j d_k) \\+\sum_{\beta<\alpha}c_{\alpha,\beta}\Big[\partial^{\alpha-\beta}(\partial_jv_j)\partial^{\beta}B_i + \partial^{\alpha-\beta}v_j\partial^{\beta}(\partial_j B_i)-\partial^{\alpha-\beta}(\partial_jv_i)\partial^{\beta}B_j\Big]=0
\end{multline*}
Here $c_{\alpha,\beta}$ are constants depending on the choice of $\alpha,\beta$, $\epsilon_{ijk}$ is the Levi-Civita symbol. From the above equality, we have that
\begin{multline*}
\partial_{t}\int|\partial^{\alpha}B_i|^2 +\int\Big[\partial_jv_j|\partial^{\alpha}B_i|^2 - 2\partial_jv_i\partial^{\alpha}B_i\partial^{\alpha}B_j + 2 \epsilon_{ijk}\partial^{\alpha}(\partial_j d_k)\partial^{\alpha}B_i\Big] \\+2\sum_{\beta<\alpha}c_{\alpha,\beta}\int\Big[\partial^{\alpha-\beta}(\partial_jv_j)\partial^{\beta}B_i + \partial^{\alpha-\beta}v_j\partial^{\beta}(\partial_j B_i)-\partial^{\alpha-\beta}(\partial_jv_i)\partial^{\beta}B_j\Big]\partial^{\alpha}B_i=0
\end{multline*}
Now, we sum up all the index $|\alpha|\leq k,\;i$, then there exist a constant $c(k)$, such that
\begin{equation}
\partial_{t}\|B\|_{H^{k}}^2 \leq c(k)M_{k+1}(\|B\|_{H^{k}}^2+1)
\end{equation}
Therefore, by Gronwall's lemma, we can get the conclusion.

(iii) We only prove the estimate for $B$. Let $\omega=B[z,B_0]-B[\tilde{z},B_0]$, then we have that
\begin{multline*}
\partial_{t}\omega_i +\partial_jv_j\omega_i + v_j\partial_j w_i-\partial_jv_i\omega_j + \epsilon_{ijk}\partial_j( d_k-\tilde{d}_k) \\ + \partial_j(v_j-\tilde{v}_j)\tilde{B}_i + (v_j-\tilde{v}_j)\partial_j \tilde{B}_i-\partial_j(v_i-\tilde{v}_i)\tilde{B}_j=0
\end{multline*}
So we have
\begin{multline*}
\partial_{t}\int|\omega_i|^2 +\int\partial_jv_j|\omega_i|^2 - 2\int\partial_jv_i\omega_i\omega_j + 2 \epsilon_{ijk}\int \partial_j( d_k-\tilde{d}_k)\omega_i \\ + 2\int\partial_j(v_j-\tilde{v}_j)\tilde{B}_i\omega_i + 2\int(v_j-\tilde{v}_j)\partial_j \tilde{B}_i\omega_i - 2\int\partial_j(v_i-\tilde{v}_i)\tilde{B}_j\omega_i=0
\end{multline*}
\begin{equation*}
\partial_{t}\|w\|_{L^2}^2 \leq c(\|z-\tilde{z}\|_{C^1} + \|v\|_{C^1} + 1)\|w\|_{L^2}^2 + c\|z-\tilde{z}\|_{C^1}(\|\tilde{B}\|_{H^1}^2+1)
\end{equation*}
Now we use the same strategy as in (ii) to compute the estimate for $D\omega$, without entering the details, we finally can get
\begin{equation}
\partial_{t}\|w\|_{H^1}^2 \leq c(\|v-\tilde{v}\|_{C^2} + \|v\|_{C^2} + 1)\|w\|_{H^1}^2 + c\|z-\tilde{z}\|_{C^2}(\|\tilde{B}\|_{H^2}^2+1)
\end{equation}
By Gronwall's lemma and\eqref{eq:ddap-e3}, this implies the inequality \eqref{eq:ddap-e5}.

\end{proof}

%%%%%%%%%%%%%%%%%%%%%%%%%%%%%%%%%%%%%%%%%

\subsection{The Faedo-Galerkin Approximate Scheme}

Now let us consider the approximate equation \eqref{eq:ddap-4} and \eqref{eq:ddap-3}. We will not try to find a solution.  Instead, we will find an approximate solution that satisfy the weak formulation of \eqref{eq:ddap-4} and \eqref{eq:ddap-3} on a finite dimensional space $X_N$, very like the Galerkin method. For the torus $\mathbb{T}^3$, we can choose $X_N=[\text{span}\{e_i,\tilde{e}_i\}_{i=1}^{N}]^3$ equipped with the 2-norm ($\dim X_N=6N$), where $e_i=\sqrt{2}\sin (2\pi \vec{k}_i\cdot x),\;\tilde{e}_i=\sqrt{2}\cos(2\pi \vec{k}_i\cdot x)$ and $\{\vec{k}_i\}_{i=1}^{\infty}$ is a permutation of $\mathbb{Z}^3_+$, where
\begin{equation*}{\small
  \mathbb{Z}^3_+:=\Big\{(n_1,n_2,n_3)\in \mathbb{Z}^3|\ n_1\geq0\Big\}\backslash \Big(\big\{(0,n_2,n_3)\in  \mathbb{Z}^3|\ n_2<0\big\}\cup\big\{(0,0,n_3)\in  \mathbb{Z}^3|\ n_3\leq 0\big\}\Big)}.
\end{equation*}

We can see that $\{e_i,\tilde{e}_i\}_{i=1}^{\infty}$ is not only the normalized orthogonal basis of $L^2(\mathbb{T}^3)$, but also the orthogonal basis of $H^k(\mathbb{T}^3),\;k\in \mathbb{N}$. Now we set $X=\bigcup_{n=1}^{\infty}X_n$, then $X$ is dense in $C^k(\mathbb{T}^3,\mathbb{R}^3),\;k\in\mathbb{N}$. Because $X_N$ has finite dimensions ($\dim X_N=6N$), then there exist a constant $c=c(N,k)$, such that
$$\|v\|_{X_N}=\|v\|_{L^2},\;\|v\|_{C^k}\leq c(N,k)\|v\|_{X_N},\;\forall v\in X_N$$

Now for every strictly positive function $\rho\in L^1(\mathbb{T}^3),\;\rho\geq\underline{\rho}>0$, we introduce a family of operators
$$\mathcal{M}_N[\rho]:\ X_{N}\mapsto X_N^{*},\ \big<\mathcal{M}_N[\rho]v,u\big>=\int_{\mathbb{T}^3}\rho v\cdot u,\;\;\forall u,v\in X_N$$
Then we have the following lemma:

\begin{lemma}\label{lm:ddfg1}
The family of operator $\mathcal{M}_N[\cdot]$ satisfies the following  properties\\
{\rm\bf (i)} For any function $\rho\in L^1(\mathbb{T}^3),\ \mathcal{M}_N[\rho]\in\mathcal{L}(X_{N},X_N^{*})$.\\
{\rm\bf(ii)} For any strictly positive function $\rho\in L^1(\mathbb{T}^3),\;\rho\geq\underline{\rho}>0$, $\mathcal{M}_N[\rho]$ is invertible, $\mathcal{M}_N^{-1}[\rho]\in\mathcal{L}(X_N^*,X_N)$, and
\begin{equation}\label{eq:ddfg-e1}
\|\mathcal{M}_N^{-1}[\rho]\|_{\mathcal{L}(X_N^*,X_N)}\leq\underline{\rho}^{-1}
\end{equation}
{\rm\bf(iii)} For any $\rho_1,\rho_2\in L^1(\mathbb{T}^3),\;\rho_1,\rho_2\geq\underline{\rho}>0$
\begin{equation}\label{eq:ddfg-e2}
\|\mathcal{M}_N^{-1}[\rho_1]-\mathcal{M}_N^{-1}[\rho_2]\|_{\mathcal{L}(X_N^*,X_N)}\leq c(N,\underline{\rho})\|\rho_1-\rho_2\|_{L^1}
\end{equation}

\end{lemma}

\begin{proof}
(i) For any $u,v\in X_N$, we have
\begin{equation}\label{eq:ddfg-e3}
\left|\big<\mathcal{M}_N[\rho]v,u\big>\right|\leq\|v\cdot u\|_{L^{\infty}}\int|\rho|\leq c(N)\|\rho\|_{L^1}\|v\|_{X_N}\|u\|_{X_N}
\end{equation}
(ii) For any $v_1,v_2\in X_N,\;v_1\neq v_2$, let $u=v_1-v_2$, then we have
$$\big<\mathcal{M}_N[\rho]v_1-\mathcal{M}_N[\rho]v_2,u\big>=\int\rho|v_1-v_2|^2>0$$
So $\mathcal{M}_N[\rho]:\ X_{N}\mapsto X_N^{*}$ is injective. Because $X_N$ has finite dimension, so $\mathcal{M}_N[\rho]:\ X_{N}\mapsto X_N^{*}$ is bijective, thus invertible. For any $\chi\in X_N$, set $v=\mathcal{M}^{-1}_N[\rho]\chi$, then
$$\|\chi\|_{X_N^*}\|v\|_{X_N}\geq\big<\chi,v\big>=\int\rho|v|^2\geq\underline{\rho}\|v\|^2_{L^2}=\underline{\rho}\|v\|^2_{X_N}$$
$$\Longrightarrow\qquad\|\mathcal{M}^{-1}_N[\rho]\chi\|_{X_N}\leq\underline{\rho}^{-1}\|\chi\|_{X_N^*}$$

(iii) Because of the identity
 $$\mathcal{M}_N^{-1}[\rho_1]-\mathcal{M}_N^{-1}[\rho_2]=\mathcal{M}_N^{-1}[\rho_1](\mathcal{M}_N[\rho_2]-\mathcal{M}_N[\rho_1])\mathcal{M}_N^{-1}[\rho_2]$$
Use the inequality in \eqref{eq:ddfg-e1} and \eqref{eq:ddfg-e3}, we can get the result.
\end{proof}

With the above properties, by standard fixed point argument, we can get the following theorem:

\begin{theorem}\label{thm:ddfg1}
For any initial data $0<h_0\in C^{\infty}(\mathbb{T}^3)$, divergence-free vector field $B_0\in C^{\infty}(\mathbb{T}^3,\mathbb{R}^3)$, $P_0,D_0\in L^2(\mathbb{T}^3,\mathbb{R}^3)$ and any $T>0$, there exists a solution $(h_n,B_n,d_n,v_n)$ with $h_n\in C^1([0,T],C^k(\mathbb{T}^3))$, $B_n\in C^1([0,T],C^k(\mathbb{T}^3,\mathbb{R}^3))$, $z_n=(d_n,v_n)\in C([0,T],X_n\times X_n)$ to the equation \eqref{eq:ddap-1}-\eqref{eq:ddap-4} in the following sense:\\
{\rm\bf(i)} \eqref{eq:ddap-1} and \eqref{eq:ddap-2} are satisfied in the classical sense, that is $h_n=h[z_n,h_0],\;B_n=B[z_n,B_0]$.\\
{\rm\bf(ii)} \eqref{eq:ddap-4} and \eqref{eq:ddap-3} are satisfied in the weak sense on $X_N$, more specifically, for any $t\in[0,T]$ and any $\psi\in X_N$,
\begin{equation}\label{eq:ddfg-s2}
\int h_n(t)d_n(t)\cdot\psi {\rm d}x- \int D_0 \cdot\psi{\rm d}x = \int_0^t\int \mathcal{S}_{\varepsilon}(h_n,B_n,d_n,v_n)\cdot\psi {\rm d}x{\rm d}s
\end{equation}
\begin{equation}\label{eq:ddfg-s1}
\int h_n(t)v_n(t)\cdot\psi {\rm d}x- \int P_0 \cdot\psi{\rm d}x = \int_0^t\int \mathcal{N}_{\varepsilon}(h_n,B_n,d_n,v_n)\cdot\psi {\rm d}x{\rm d}s
\end{equation}
where
$$\mathcal{S}_{\varepsilon}(h,B,d,v)= -\nabla\cdot[h(d\otimes v- v\otimes d)] - (-\triangle)^{l}d +\varepsilon^{-1}\big[\nabla\times b -hd\big]$$
$$\mathcal{N}_{\varepsilon}(h,B,d,v)= (h d\cdot\nabla)d - \nabla\cdot (hv\otimes v) -(-\triangle)^{l}v + \varepsilon^{-1}\left[\nabla\cdot \left(\frac{B\otimes B}{h}\right) + \nabla\left(h^{-1}\right)-hv\right]$$

\end{theorem}

\begin{proof}
{\bf Step\ 1:\ Local Existence}\newline

Let us define a map $J_n[\cdot]:\ L^2(\mathbb{T}^3,\mathbb{R}^3)\mapsto X_n^*$, such that
$$\big<J_n[f],\psi\big>=\int f\cdot\psi,\;\forall\psi\in X_n.$$ $J_n$ can be seen as the orthogonal projection of $L^2(\mathbb{T}^3,\mathbb{R}^3)$ onto $X_n^*$. We have
\begin{equation*}
\|J_n[f]\|_{X_n^*}\leq\|f\|_{L^2}
\end{equation*}
Let us consider the operator $K_n[\cdot]:\ C([0,T],X_n\times X_n)\mapsto C([0,T],X_n\times X_n)$, which maps $z=(d,v)\mapsto K_n[z]=(K_n[z]^d,K_n[z]^v)$ such that
$$K_n[z]^d(t)=\mathcal{M}^{-1}_n[h[z,h_0](t)]\left(J_{n}\left[D_0 + \int_0^t \mathcal{S}_{\varepsilon}(h[z,h_0],B[z,B_0],d,v)(s){\rm d}s\right]\right)$$
$$K_n[z]^v(t)=\mathcal{M}^{-1}_n[h[z,h_0](t)]\left(J_{n}\left[P_0 + \int_0^t \mathcal{N}_{\varepsilon}(h[z,h_0],B[z,B_0],d,v)(s){\rm d}s\right]\right)$$
To prove the theorem, it suffices to show that $K_n$ has a fixed point on $C([0,T],X_n\times X_n)$. First, we should prove that $K_n$ is well defined. It is obvious that $K_n[z](t)\in X_n\times X_n$, $\forall t\in[0,T]$. We only need to prove the continuity. In fact, we can prove that $K_n[z](t)$ is Lipschitz continuous in time on the set
$$F_{\kappa,\sigma}=\big\{(d,v)\in C([0,\sigma],X_n\times X_n)\big|\|v(t)\|_{X_n},\|d(t)\|_{X_n}\leq \kappa,\;\forall t\in [0,\sigma]\big\},\quad 0<\sigma\leq T.$$
In fact, for any $s,t\in[0,\sigma]$,
\begin{equation*}
\begin{array}{r@{}l}
& \qquad \displaystyle{\big\|K_n[z]^v(t)-K_n[z]^v(s)\big\|_{X_n}}\\
\leq  & \displaystyle{\ \left\|\Big(\mathcal{M}^{-1}_n[h[z,h_0](t)]-\mathcal{M}^{-1}_n[h[z,h_0](s)]\Big)\left(J_{n}\left[P_0 + \int_0^s \mathcal{N}_{\varepsilon}(h,B,d,v)(r){\rm d}r\right]\right)\right\|_{X_n}}\\
& \quad \displaystyle{+\ \  \left\|\mathcal{M}^{-1}_n[h[z,h_0](t)]\left(J_{n}\left[\int_s^t \mathcal{N}_{\varepsilon}(h[z,h_0],B[z,B_0],d,v)(r){\rm d}r\right]\right)\right\|_{X_n} }
\end{array}
\end{equation*}
Now, we let
\begin{equation}\label{eq:ddfg-con}
C_0=\max\big\{(\inf h_0)^{-1},\sup h_0,\|h_0\|_{H^4},\|B_0\|_{H^4},\|P_0\|_{L^2},\|D_0\|_{L^2}\big\}
\end{equation}
By \eqref{eq:ddap-e1},\eqref{eq:ddap-e2},\eqref{eq:ddap-e3},\eqref{eq:ddfg-e1},\eqref{eq:ddfg-e2}, we have
\begin{equation*}
\begin{array}{r@{}l}
\displaystyle{
\big\|K_n[z]^v(t)-K_n[z]^v(s)\big\|_{X_n}} &
\leq \displaystyle{\ \ c(T,n,\varepsilon,\kappa,C_0)\left[\int\int_s^t\big|\partial_th[z,h_0](r)\big|{\rm d}r{\rm d}x + |t-s|\right]}\\
&\leq  \displaystyle{\ \ c(T,n,\varepsilon,\kappa,C_0) |t-s|}
\end{array}
\end{equation*}
By the same reason,
\begin{equation*}
\big\|K_n[z]^d(t)-K_n[z]^d(s)\big\|_{X_n}
\leq c(T,n,\varepsilon,\kappa,C_0) |t-s|
\end{equation*}
Therefore, we have that
$$K_n[\cdot]:\ F_{\kappa,\sigma}\mapsto W^{1,\infty}([0,\sigma],X_n\times X_n)\subset\subset C([0,\sigma],X_n\times X_n).$$
Moreover, we have
$$\big\|K_n[z]^v(t)\big\|_{X_n}\leq \|\mathcal{M}^{-1}_n[h_0]J_n[P_0]\|_{X_n} + c(T,n,\varepsilon,\kappa,C_0) t$$
$$\big\|K_n[z]^d(t)\big\|_{X_n}\leq \|\mathcal{M}^{-1}_n[h_0]J_n[D_0]\|_{X_n} + c(T,n,\varepsilon,\kappa,C_0) t$$

Now if we choose $\kappa=\kappa_0$ big enough and $\sigma=\sigma_0$ small enough, for example
\begin{equation}\label{eq:ddfg-kt}
\kappa_0\geq 2(\inf h_0)^{-1}\big(\|P_0\|_{L^2}+\|D_0\|_{L^2}\big),\;\;\sigma_0\leq\frac{\kappa_0}{2c(T,n,\varepsilon,\kappa_0,C_0)}
\end{equation}
Then we have that $K_n[\cdot]:\ F_{\kappa_0,\sigma_0}\hookrightarrow F_{\kappa_0,\sigma_0}$. By Arzel\'a-Ascoli theorem, $K_n[F_{\kappa_0,\sigma_0}]$ is a relatively compact subset of $F_{\kappa_0,\sigma_0}$. Now if we can prove that $K_n$ is a continuous map, then by Schauder's Fixed Point Theorem, there exist a fixed point $z$ of $K_n$ on $F_{\kappa_0,\sigma_0}$, such that $z=K_n[z]$. Now, let's show that the map $K_n$ is continuous on $F_{\kappa_0,\sigma_0}$. For $z,\tilde{z}\in F_{\kappa_0,\sigma_0}$, $\forall t\in[0,\sigma_0]$, we have,
\begin{equation*}
\begin{array}{r@{}l}
& \qquad \displaystyle{\big\|K_n[z]^v(t)-K_n[\tilde{z}]^v(t)\big\|_{X_n}}\\
\leq  & \displaystyle{\ \ \left\|\big[\mathcal{M}^{-1}_n[h[z,h_0](t)]-\mathcal{M}^{-1}_n[h[\tilde{z},h_0](t)]\big]\left(J_{n}\left[P_0 + \int_0^t \mathcal{N}_{\varepsilon}(h,B,d,v)(\tau){\rm d}\tau\right]\right)\right\|_{X_n}}\\
 &  \displaystyle{+\ \left\|\mathcal{M}^{-1}_n[h[\tilde{z},h_0](t)]\left(J_{n}\left[\int_0^t \Big(\mathcal{N}_{\varepsilon}(h,B,d,v)(\tau)-\mathcal{N}_{\varepsilon}(\tilde{h},\tilde{B},\tilde{d},\tilde{v})(\tau)\Big){\rm d}\tau\right]\right)\right\|_{X_n} }\\
\leq  & \displaystyle{\quad c(T,n,\varepsilon,\kappa_0,C_0)\big\|h[z,h_0](t)-h[\tilde{z},h_0](t)\big\|_{L^1}}\\
& \ \displaystyle{+\  c(T,n,\varepsilon,\kappa_0,C_0) \Big\|\mathcal{N}_{\varepsilon}(h[z,h_0],B[z,B_0],d,v) -\mathcal{N}_{\varepsilon}(h[\tilde{z},h_0],B[\tilde{z},B_0],\tilde{d},\tilde{v})\Big\|_{L^2_{t,x}}}
\end{array}
\end{equation*}
By \eqref{eq:ddap-e4} and \eqref{eq:ddap-e5}, we can finally have that
\begin{equation*}
\begin{array}{r@{}l}
& \qquad \displaystyle{\sup_{t\in[0,T]}\big\|K_n[z]^v(t)-K_n[\tilde{z}]^v(t)\big\|_{X_n}}\\
\leq  & \displaystyle{\ \ c(T,n,\varepsilon,\kappa_0,C_0)\Big(\big\|v-\tilde{v}\big\|_{C([0,\sigma_0],X_n)}+\big\|d-\tilde{d}\big\|_{C([0,\sigma_0],X_n)}\Big)}
\end{array}
\end{equation*}
By similar argument, we can show that $K_n$ is continuous on $F_{\kappa_0,\sigma_0}$. Then by Schauder's Fixed Point Theorem, there is a ``solution" on $[0,\sigma_0]$.\newline

{\bf Step\ 2:\ Global Existence}\newline

From the above argument, we know that at small time interval $[0,\sigma_0]$ there exist a solution $z\in F_{\kappa_0,\sigma_0}$ s.t. $z=K_n[z]$. Now we want to apply the fixed point argument repeatedly to obtain the existence on the whole time interval $[0,T]$. Now we suppose that on $[0,T_0], T_0<T$, we have a fixed point $z=K_n[z]$, we use $\tilde{h}_0=h[z,h_0](T_0)$, $\tilde{B}_0=B[z,B_0](T_0)$, $\tilde{D}_0=\tilde{h}_0d(T_0)$, $\tilde{P}_0=\tilde{h}_0v(T_0)$ as our new initial data, and use the local existence result above to extend the existence interval. This argument can be applied as long as we can prove that there is still a fixed point in $F_{\kappa_0,\tau_0}$ with our new initial data $(\tilde{h}_0,\tilde{B}_0,\tilde{D}_0,\tilde{P}_0)$ while the constants $\kappa_0,\sigma_0$ chosen in previous part do not change. (They only depend on the initial data $h_0,B_0,D_0,P_0,\varepsilon, T$ and $n$). Now, we only need to prove that the constant $\tilde{C}_0$ for the new initial data defined in \eqref{eq:ddfg-con} has a uniform bound.

Suppose $(h_n,B_n,d_n,v_n)$ is the solution on $[0,T_0]$. Because  $d_n,v_n\in C([0,T_0],X_n)$ is Lipschitz continuous, so it is differentiable almost everywhere. We take the derivative on both sides of \eqref{eq:ddfg-s2} and \eqref{eq:ddfg-s1}, then we have that, for any $\varphi,\psi\in X_n$, any $g\in C^1(\mathbb{T}^3),\;\phi\in C^1(\mathbb{T}^3,\mathbb{R}^3)$,
\begin{equation*}
\int \partial_th_n(t)g   - \int h_n(t)v_n(t)\cdot\nabla g =0
\end{equation*}
\begin{equation*}
\int \partial_t B_n(t)\cdot\phi  - \int (B_n\otimes v_n-v_n\otimes B_n):\nabla\phi + \int d_n\cdot(\nabla\times\phi)=0
\end{equation*}
\begin{multline*}
\int \partial_t\big(h_n(t)d_n(t)\big)\cdot\psi  - \int h_n(d_n\otimes v_n-v_n\otimes d_n):\nabla\psi +\int\nabla^ld_n:\nabla^l\psi \\ + \varepsilon^{-1}\int\big[-b_n\cdot(\nabla\times\psi) + h_nd_n\cdot\psi \big]=0
\end{multline*}
\begin{multline*}
\int \partial_t\big(h_n(t)v_n(t)\big)\cdot\varphi -\int h_nv_n\otimes v_n:\nabla\varphi -\int\big[(h_n d_n\cdot\nabla) d_n\big]\cdot\varphi \\+ \int\nabla^lv_n:\nabla^l\varphi   + \varepsilon^{-1}\int\left[h_n^{-1}(B_n\otimes B_n +I_3):\nabla\varphi + h_nv_n\cdot\varphi\right]=0
\end{multline*}

Now, let's choose $\phi=b_n(t)=B_n(t)/h_n(t),\;\psi=\varepsilon d_n(t),\;\varphi=\varepsilon v_n(t)$ and $g=-\frac{1}{2}\big(|h_n(t)|^{-2} +|b_n(t)|^2+\varepsilon|v_n(t)|^2+\varepsilon|d_n(t)|^2\big)$, then add the these equations together, we have the following equality
\begin{multline}\label{eq:ddfg-en1}
\frac{{\rm d}}{{\rm d}t}\int\frac{h_n(t)}{2}\Big(|h_n(t)|^{-2} +|b_n(t)|^2+\varepsilon|v_n(t)|^2+\varepsilon|d_n(t)|^2\Big)\\ + \int h_n(t)(|v_n(t)|^2+|d_n(t)|^2)  +\varepsilon\int\left[|\nabla^{l}v_n(t)|^2+|\nabla^{l}d_n(t)|^2\right]=0
\end{multline}
We denote $\displaystyle{\Lambda_n(t)=\int\frac{h_n(t)}{2}\Big(|h_n(t)|^{-2} +|b_n(t)|^2+\varepsilon|v_n(t)|^2+\varepsilon|d_n(t)|^2\Big)}$, then from the above equality, we have
\begin{equation}\label{eq:ddfg-enco1}
\sup_{t\in[0,T_0]}\Lambda_n(t),\; \|h^{\frac{1}{2}}_n v_n\|^2_{L^2_{t,x}} ,\; \|h^{\frac{1}{2}}_n d_n\|^2_{L^2_{t,x}} ,\; \varepsilon\|\nabla^{l}v_n\|_{L^2_{t,x}}^2,\;\varepsilon\|\nabla^{l}d_n\|_{L^2_{t,x}}^2\leq \Lambda_n(0)
\end{equation}
By Cauchy-Schwartz inequality, we can easily get that
$$\|d_n\|_{L^{2}(L^1)},\|v_n\|_{L^{2}(L^1)}\leq \sqrt{2}\Lambda_n(0),\;\;\;\|\nabla^lv_n\|_{L^{2}(L^1)}\leq\varepsilon^{-1}\Lambda_n(0)$$
So we get that $\|v_n\|_{L^2(W^{l,1})}\leq c(\varepsilon,\Lambda_n(0))$. With $l>4$, by Sobolev embedding, we have that $\|v_n\|_{L^2(W^{1,\infty})}\leq c(\varepsilon,\Lambda_n(0))$. So by the result in \lmref{lm:ddap2}.(i), we have that, there exist a constant $c_0=c_0(\varepsilon,T,\Lambda_n(0))>0$ s.t. for all $t\in[0,T_0],\;x\in\mathbb{T}^3$,
$$
0<c_0\leq h_n(t,x)\leq c_0^{-1}
$$
So by \eqref{eq:ddfg-enco1}, we have
\begin{equation*}
\sup_{t\in[0,T_0]}\|v_n\|^2_{L^2},\;\sup_{t\in[0,T_0]}\|d_n\|^2_{L^2},\;\sup_{t\in[0,T_0]}\|P_n\|^2_{L^2},\;\sup_{t\in[0,T_0]}\|D_n\|^2_{L^2}\leq 2(\varepsilon c_0)^{-1}\Lambda_n(0)
\end{equation*}
Then by \lmref{lm:ddap2}.(ii), we have that
$$\sup_{t\in[0,T_0]}\|h\|_{H^4},\sup_{t\in[0,T_0]}\|B\|_{H^4}\leq c(T,n,\varepsilon,\Lambda_n(0),\|h_0\|_{H^4},\|B_0\|_{H^4})$$
So we know that $\tilde{C}_0$ have a uniform bound that depends only on initial data. So we can see from the choice of $\kappa_0,\sigma_0$ in \eqref{eq:ddfg-kt} that by slightly modify the choice of $\kappa_0,\sigma_0$, the fixed point method can be repeatedly applied on the same space $F_{\kappa_0,\sigma_0}$, so we get the global existence on $[0,T]$.

\end{proof}

%%%%%%%%%%%%%%%%%%%%%%%%%%%%%%%%%%%%%%%%%%%%
%%%%%%%%%%%%%%%%%%%%%%%%%%%%%%%%%%%%%%%%%%%%
%%%%%%%%%%%%%%%%%%%%%%%%%%%%%%%%%%%%%%%%%%%%

\section{Existence of the Dissipative Solution}\label{sec:7}

%%%%%%%%%%%%%%%%%%%%%%%%%%%%%%%%%%%%%%%%%%%%

\subsection{Smooth approximation of initial data}

We suppose that our initial data $h_{0}$ is a nonnegative Borel measure in $C(\mathbb{T}^3,\mathbb{R})'$, $B_{0}\in C(\mathbb{T}^3,\mathbb{R}^3)'$, satisfying $\nabla\cdot B_{0}=0$ in the sense of distributions. Moreover, we suppose $0<\Lambda(h_{0},U_{0})<\infty$, where $U_{0}=\big(\mathcal{L},B_{0}\big)$. Now we will find a family of smooth functions to approach our initial data.

Let us define a positive Schwartz function $\widetilde{\rho}(x)=\frac{1}{(2\pi)^{\frac{3}{2}}}e^{-\frac{|x|^2}{2}}\in C^{\infty}(\mathbb{R}^3,\mathbb{R})$. We have that $\int_{\mathbb{R}^{3}}\widetilde{\rho}(x){\rm d}x=1$. For any $0<\varepsilon<1$, we define a function $\rho_{\varepsilon}$ on $\mathbb{T}^{3}$ by
\begin{equation}
\rho_{\varepsilon}(x)=\sum_{\vec{k}\in\mathbb{Z}^{3}}\widetilde{\rho}_{\varepsilon}(x+\vec{k})=\sum_{\vec{k}\in\mathbb{Z}^{3}}\frac{1}{\varepsilon^{3}}
\widetilde{\rho}\left(\frac{x+\vec{k}}{\varepsilon}\right).
\end{equation}
We can easily check that $\rho_{\varepsilon}(x)$ is also a smooth positive function on $\mathbb{T}^3$, and we have $\int_{\mathbb{T}^{3}}\rho_{\varepsilon}(x){\rm d}x=1$. Now, for $0<\varepsilon<1$, we define
\begin{equation}
h_{0}^{\varepsilon}=h_{0}\ast\rho_{\varepsilon}=\int_{\mathbb{T}^3}\rho_{\varepsilon}(x-y){\rm d}\:h_{0}(y),\;\;\;B_{0}^{\varepsilon}=B_{0}\ast\rho_{\varepsilon}
\end{equation}
Because $0<\Lambda(h_{0},U_{0})<\infty$, then $h_{0}\geq0,\;h_{0}\neq0$. So we have that $h_{0}^{\varepsilon}>0$ for any $0<\varepsilon<1$. Besides, it's easily to verify that $B_{0}^{\varepsilon},h_{0}^{\varepsilon}$ are smooth functions on $\mathbb{T}^3$ and converge to $B_{0},h_{0}$ in the weak-$*$ topology of $C(\mathbb{T}^3)'$. Moreover, for any smooth function $\phi$ on $\mathbb{T}^3$, we have
\begin{equation}
\begin{array}{r@{}l}
\displaystyle{ \int_{\mathbb{T}^3} \phi\nabla \cdot B^{\varepsilon}_{0} } &  \displaystyle{  = - \int_{\mathbb{T}^3}\nabla\phi(x)\cdot\left(\int_{\mathbb{T}^3}\rho_{\varepsilon}(x-y){\rm d}B_{0}(y)\right) {\rm d}x }\\
& \displaystyle{  = \int_{\mathbb{T}^3}\nabla_{y}\left(\int_{\mathbb{T}^3}\phi(x)\rho_{\varepsilon}(x-y){\rm d}x\right){\rm d}B_{0}(y) = 0}
\end{array}
\end{equation}
So we know that
\begin{equation}\label{eq:ddpf-inapb2}
\nabla \cdot B^{\varepsilon}_{0}=0.
\end{equation}
Besides, we can get the following result.

\begin{proposition}\label{prop:ddpf1}
For all $0<\varepsilon<1$, we have that
\begin{equation}\label{eq:ddpf-inap}
\Lambda(h^{\varepsilon}_{0},U^{\varepsilon}_{0})=\int\frac{|B^{\varepsilon}_{0}|^{2}+1}{2h^{\varepsilon}_{0}}\leq\Lambda(h_{0},U_{0})
\end{equation}
Moreover, we have that
$$\Lambda(h^{\varepsilon}_{0},U^{\varepsilon}_{0})\rightarrow\Lambda(h_{0},U_{0})\quad {as}\quad \varepsilon\rightarrow 0.$$

\end{proposition}

\begin{proof}

We know that
\begin{equation*}
\begin{array}{r@{}l}
& \quad\Lambda(h^{\varepsilon}_{0},U^{\varepsilon}_{0})\\
= & \displaystyle{\sup_{a\in C(\mathbb{T}^3,\mathbb{R}), A\in C(\mathbb{T}^3,\mathbb{R}^{4})\atop a+\frac{1}{2}|A|^2\leq 0}\int a(x) \left(\int\rho_{\varepsilon}(x-y){\rm d}h_{0}(y)\right){\rm d}x + \int A(x) \left(\int\rho_{\varepsilon}(x-y){\rm d}U_{0}(y)\right){\rm d}x }\\
= & \displaystyle{\sup_{a\in C(\mathbb{T}^3,\mathbb{R}), A\in C(\mathbb{T}^3,\mathbb{R}^{4})\atop a+\frac{1}{2}|A|^2\leq 0}\int \left(\int a(x)\rho_{\varepsilon}(x-y){\rm d}x\right){\rm d}h_{0}(y) + \int \left(\int A(x)\rho_{\varepsilon}(x-y){\rm d}x\right){\rm d}U_{0}(y) }
\end{array}
\end{equation*}
By Cauchy-Schwarz inequality, we can easily know that
\begin{equation*}
\begin{array}{r@{}l}
\displaystyle{ \left|\int A(x)\rho_{\varepsilon}(x-y){\rm d}x\right|^2 } & \displaystyle{ \leq\left(\int |A(x)|^2\rho_{\varepsilon}(x-y){\rm d}x\right)\left(\int \rho_{\varepsilon}(x-y){\rm d}x\right)} \\
& \displaystyle{\leq-2\int a(x)\rho_{\varepsilon}(x-y){\rm d}x}
\end{array}
\end{equation*}
So we get that
$$\Lambda(h^{\varepsilon}_{0},U^{\varepsilon}_{0}) \leq \sup_{\tilde{a}\in C(\mathbb{T}^3,\mathbb{R}), \tilde{A}\in C(\mathbb{T}^3,\mathbb{R}^{4})\atop \tilde{a}+\frac{1}{2}|\tilde{A}|^2\leq 0} \big\langle h_{0},\tilde{a}\big\rangle + \big\langle U_{0},\tilde{A}\big\rangle=\Lambda(h_{0},U_{0})$$
Now because for each fixed continuous function $a,A$, $\big\langle h^{\varepsilon}_{0},a\big\rangle\rightarrow\big\langle h_{0},a\big\rangle,\;\big\langle U^{\varepsilon}_{0},A\big\rangle\rightarrow\big\langle U_{0},B\big\rangle$ as $\varepsilon\rightarrow0$, then we have that
$$\liminf_{\varepsilon\rightarrow0}\Lambda(h^{\varepsilon}_{0},U^{\varepsilon}_{0}) \geq\Lambda(h_{0},U_{0})$$
Combining the above two results, we can get the convergence $\Lambda(h^{\varepsilon}_{0},U^{\varepsilon}_{0})\rightarrow\Lambda(h_{0},U_{0})$ as $\varepsilon\rightarrow0$.

\end{proof}

%%%%%%%%%%%%%%%%%%%%%%%%%%%%%%%%%%%%%%%%%%%%%%%%%

\subsection{Existence of converging sequence}

Now let $\big\{\varepsilon_k\big\}_{k=1}^{\infty}$ be a sequence such that $0<\varepsilon_k<1$, $\displaystyle{\lim_{k\rightarrow\infty}\varepsilon_k=0}$. By \thmref{thm:ddfg1}, for every $n\in\mathbb{N}^*$, there exists a solution $(h_n^{\varepsilon_k},B_n^{\varepsilon_k},d_n^{\varepsilon_k},v_n^{\varepsilon_k})$ on $[0,T]$ that satisfies \eqref{eq:ddap-1},\eqref{eq:ddap-2},\eqref{eq:ddfg-s2},\eqref{eq:ddfg-s1} with $\varepsilon=\varepsilon_k$ and the initial data  $(h_0^{\varepsilon_k},\;B_0^{\varepsilon_k},\;0,\;0)$. For simplicity, we denote
$$B_n^{\varepsilon_k}=h_n^{\varepsilon_k}b_n^{\varepsilon_k},\;\;\;\;D_n^{\varepsilon_k}=h_n^{\varepsilon_k}d_n^{\varepsilon_k},\;\;\;\;
P_n^{\varepsilon_k}=h_n^{\varepsilon_k}v_n^{\varepsilon_k},$$
$$\Lambda_n^{\varepsilon_k}(t)=\int\frac{h_n^{\varepsilon_k}}{2}\Big(\big(h_n^{\varepsilon_k}\big)^{-2}+\big|b_n^{\varepsilon_k}\big|^2+
\varepsilon_k\big|d_n^{\varepsilon_k}\big|^2+\varepsilon_k\big|v_n^{\varepsilon_k}\big|^2\Big),\;\;\;\Lambda^{\varepsilon_k}_0=\int\frac{|B_0^{\varepsilon_k}\big|^2
+1}{2h_0^{\varepsilon_k}}$$

\begin{lemma}\label{lm:ddpf-cov1}
Suppose $(h_n^{\varepsilon_k},B_n^{\varepsilon_k},d_n^{\varepsilon_k},v_n^{\varepsilon_k})$ is the solution in \thmref{thm:ddfg1} with initial data $(h^{\varepsilon_k}_0,B^{\varepsilon_k}_0,0,0)$. Then there exist a constant $C_0$ that depends only on $h_0,B_0$, such that for all $n$ and $\varepsilon_k$,
\begin{equation}\label{eq:ddpf-covc1}
\big\|h_n^{\varepsilon_k}\big\|_{L^{\infty}_t(L^1_x)},\big\|h_n^{\varepsilon_k}b_n^{\varepsilon_k}\big\|_{L^{\infty}_t(L^1_x)},
\big\|h_n^{\varepsilon_k}d_n^{\varepsilon_k}\big\|_{L^{2}_t(L^1_x)},\big\|h_n^{\varepsilon_k}v_n^{\varepsilon_k}\big\|_{L^{2}_t(L^1_x)}\leq C_0
\end{equation}
\begin{equation}
\sqrt{\varepsilon_k}\big\|\nabla^ld_n^{\varepsilon_k}\big\|_{L^{2}_{t,x}},\sqrt{\varepsilon_k}\big\|\nabla^lv_n^{\varepsilon_k}\big\|_{L^{2}_{t,x}}\leq C_0
\end{equation}
\end{lemma}

\begin{proof}
By \lmref{lm:ddap2}, we know that $h_n^{\varepsilon_k}$ is always positive. Since $h_n^{\varepsilon_k}$ solves \eqref{eq:ddap-1}, we have
$$\int h_n^{\varepsilon_k}(t)=\int h_0^{\varepsilon_k}=\int\left(\int\rho_{\varepsilon_k}(x-y){\rm d}h_{0}(y)\right){\rm d}x=\int h_0$$
By \eqref{eq:ddfg-enco1} and \eqref{eq:ddpf-inap}, we know that
\begin{multline}
\sup_{t\in[0,T]}\int\frac{1 +|B^{\varepsilon_k}_n(t)|^2}{2 h^{\varepsilon_k}_n(t)} + \int^{T}_{0}\int h^{\varepsilon_k}_n(t)(|v^{\varepsilon_k}_n(t)|^2+|d^{\varepsilon_k}_n(t)|^2) \\ +\varepsilon_k\int^{T}_{0}\int\left[|\nabla^{l}v^{\varepsilon_k}_n(t)|^2+|\nabla^{l}d^{\varepsilon_k}_n(t)|^2\right]\leq \int\frac{1 +|B^{\varepsilon_k}_0|^2}{2 h^{\varepsilon_k}_0} \leq \Lambda(h_0,U_0)
\end{multline}
By Cauchy-Schwartz inequality,
$$\left(\int\big|B_n^{\varepsilon_k}(t)\big|\right)^2\leq \left(\int\frac{|B^{\varepsilon_k}_n(t)|^2}{h_n^{\varepsilon_k}(t)}\right)\left(\int h_n^{\varepsilon_k}(t)\right)\leq4\Lambda(h_0,U_0)\int h_0$$
$$\int^{T}_{0}\left(\int\big|h_n^{\varepsilon_k}(t)v_n^{\varepsilon_k}(t)\big|\right)^2\leq\int^{T}_{0}\int h_n^{\varepsilon_k}(t)\big|v_n^{\varepsilon_k}(t)\big|^2\int h_n^{\varepsilon_k}(t)\leq\Lambda(h_0,U_0)\int h_0$$
We can get the conclusion easily from the above estimates.

\end{proof}

From the above lemma, we know that $(h_n^{\varepsilon_k},B_n^{\varepsilon_k},D_n^{\varepsilon_k},P_n^{\varepsilon_k})$ are bounded in some suitable spaces, so we can extract a converging subsequence.

\begin{lemma}\label{lm:ddpf-cov2}
There exists a subsequence $\big\{n_i\big\}_{i=1}^{\infty}\subseteq\mathbb{N}^*$, $h\in C([0,T],C(\mathbb{T}^3,\mathbb{R})_{w^{*}}')$, $B\in C([0,T],C(\mathbb{T}^3,\mathbb{R}^3)_{w^{*}}')$, $P,D\in C([0,T]\times\mathbb{T}^3,\mathbb{R}^3)'$, such that
$$h_{n_i}^{\varepsilon_{n_i}}\rightarrow h\;\text{in}\;C([0,T],C(\mathbb{T}^3,\mathbb{R})_{w^{*}}'),\;\;\;B_{n_i}^{\varepsilon_{n_i}}\rightarrow B\;\text{in}\;C([0,T],C(\mathbb{T}^3,\mathbb{R}^3)_{w^{*}}')$$
$$h_{n_i}^{\varepsilon_{n_i}}\xrightarrow{w^*} h\;\text{in}\;C([0,T]\times\mathbb{T}^3,\mathbb{R})'$$
$$B_{n_i}^{\varepsilon_{n_i}}\xrightarrow{w^*} B,\;P_{n_i}^{\varepsilon_{n_i}}\xrightarrow{w^*} P,\;D_{n_i}^{\varepsilon_{n_i}}\xrightarrow{w^*} D\;\text{in}\;C([0,T]\times\mathbb{T}^3,\mathbb{R}^3)'$$
Moreover, we have that $(h,B)$ is bounded in $C^{0,\frac{1}{2}}([0,T],C(\mathbb{T}^3,\mathbb{R}^4)'_{w^*})$ by some constant that depends only on $T$ and $(h_0,B_0)$.

\end{lemma}

\begin{proof}

For any smooth function $f\in C^{\infty}(\mathbb{T}^3,\mathbb{R})$, we have
\begin{equation}\label{eq:ddpf-cove1}
\begin{array}{r@{}l}
& \displaystyle{ \quad \left|\int_{\mathbb{T}^3}\big(h^{\varepsilon_k}_{n}(t,x)-h^{\varepsilon_k}_{n}(s,x)\big)f(x){\rm d}x\right|} \\
= & \displaystyle{\; \left|\int^{t}_{s}\int_{\mathbb{T}^3} P^{\varepsilon_k}_{n}(\sigma,x)\cdot\nabla f(x){\rm d}\sigma{\rm d}x\right| }\\
 \leq & \displaystyle{\;  \left(\int^{t}_{s}\int_{\mathbb{T}^3} \frac{|P^{\varepsilon_k}_{n}(\sigma,x)|^{2}}{h^{\varepsilon_k}_{n}(\sigma,x)}{\rm d}\sigma{\rm d}x\right)^{\frac{1}{2}}\left(\int^{t}_{s}\int_{\mathbb{T}^3}|\nabla f(x)|^{2}h^{\varepsilon_k}_{n}(\sigma,x){\rm d}\sigma{\rm d}x\right)^{\frac{1}{2}} }\\
 \leq & \displaystyle{\; \|\nabla f\|_{\infty}\Big(\Lambda(h_{0},U_{0})\big\langle h_{0},1 \big\rangle\Big)^{\frac{1}{2}}|t-s|^{\frac{1}{2}} }
\end{array}
\end{equation}

Besides, for any smooth function $\phi\in C(\mathbb{T}^3,\mathbb{R}^3)$
\begin{equation}\label{eq:ddpf-cove2}
\begin{array}{r@{}l}
& \displaystyle{\quad \left|\int_{\mathbb{T}^3}\big(B^{\varepsilon_k}_{n}(t,x)-B^{\varepsilon_k}_{n}(s,x)\big)\cdot \phi(x){\rm d}x\right|} \\
= & \displaystyle{\; \left|\int^{t}_{s}\int_{\mathbb{T}^3}\big(B^{\varepsilon_k}_{n} \times v^{\varepsilon_k}_{n} + d^{\varepsilon_k}_{n}\big)\cdot(\nabla\times\phi) \right| }\\
\leq & \displaystyle{\; \sqrt{2}\|\nabla\times\phi\|_{\infty}\left(\int^{t}_{s}\int_{\mathbb{T}^3} \frac{|P^{\varepsilon_k}_n|^{2}+|D^{\varepsilon_k}_n|^{2}}{h^{\varepsilon_k}_n}\right)^{\frac{1}{2}}\left(\int^{t}_{s}\int_{\mathbb{T}^3} \frac{|B^{\varepsilon_k}_n|^{2}+1}{h^{\varepsilon_k}_n}\right)^{\frac{1}{2}}}\\
\leq & \displaystyle{\; 2\|\nabla\times\phi\|_{\infty}\Lambda(h_{0},U_{0})|t-s|^{\frac{1}{2}} }
\end{array}
\end{equation}
From \eqref{eq:ddpf-covc1}, we can easily know that, for all $\varepsilon_k,\;n$ and $t$ ,the total variation of $(h^{\varepsilon_k}_n,B^{\varepsilon_k}_n)$ is bounded. By Banach-Alaoglu theorem, the closed ball $B_{R}(0)$ in $C(\mathbb{T}^3,\mathbb{R}^{4})'$ is compact with respect to the weak-$*$ topology. From \eqref{eq:ddpf-cove1},\eqref{eq:ddpf-cove2}, we know that $\{(h^{\varepsilon_n}_n,B^{\varepsilon_n}_n)\}_{n=1}^{\infty}$ is uniformly bounded in $C^{0,\frac{1}{2}}([0,T],C(\mathbb{T}^3,\mathbb{R}^{4})'_{w^{*}})$ by some constant that depends only on $T$ and $(h_0,B_0)$. So by Arzel\`{a}-Ascoli's theorem, we can extract a subsequence $\{(h^{\varepsilon_{n_i}}_{n_i},B^{\varepsilon_{n_i}}_{n_i})\}_{i=1}^{\infty}$ that converge to some measures denoted by $(h,B)$ in $C([0,T],C(\mathbb{T}^3,\mathbb{R}^{4})'_{w^{*}})$, and $(h,B)$ is bounded in $C^{0,\frac{1}{2}}([0,T],C(\mathbb{T}^3,\mathbb{R}^{4})'_{w^{*}})$ by the same constant. Besides, from \eqref{eq:ddpf-covc1}, we know that the total variations of $P^{\varepsilon_n}_n,D^{\varepsilon_n}_n$ is uniformly bounded in $C([0,T]\times\mathbb{T}^3,\mathbb{R}^3)'$, so we can extract a sub sequence that weakly converge to $P,D\in C([0,T]\times\mathbb{T}^3,\mathbb{R}^3)'$. So we get the conclusion.

\end{proof}

\subsection{The limit is a dissipative solution}

By \lmref{lm:ddpf-cov2}, we can extract a subsequence $(h^{\varepsilon_{n_i}}_{n_i},B^{\varepsilon_{n_i}}_{n_i},D^{\varepsilon_{n_i}}_{n_i},P^{\varepsilon_{n_i}}_{n_i})$ that converge strongly to a function $(h,B)$ in $C([0,T],C(\mathbb{T}^3,\mathbb{R}^{4})'_{w^{*}})$ and weakly-$*$ to $D,P$ in $C([0,T]\times\mathbb{T}^3,\mathbb{R}^3)'$. But it is not clear if these functions are dissipative solutions or not. In the following part, we will prove that $(h,B,D,P)$ satisfies all the requirements in \defref{def:dddf1}, thus it is indeed a dissipative solution of (DMHD) with the initial data $(h_{0},B_{0})$.

Firstly, we know that $(h^{\varepsilon_{n_i}}_{n_i},B^{\varepsilon_{n_i}}_{n_i})|_{t=0}=(h^{\varepsilon_{k_i}}_{0},B^{\varepsilon_{n_i}}_{0})$ converge weakly-$*$ to $(h_{0},B_{0})$, so $(h,B)|_{t=0}=(h_{0},B_{0})$.

Secondly, for any
$u\in C^1([0,T]\times\mathbb{T}^3,\mathbb{R})$ and $t\in[0,T]$, the limit $h,P$ satisfy \eqref{eq:dddf-1}. To prove this, for any $\delta>0$, let's find a smooth non increasing function on $[0,T]$ denoted by $\Theta_{\delta}$, such that $\Theta_{\delta}(s)=1,\;s\in[0,t]$, $\Theta_{\delta}(s)=0,\;s\in[t+\delta,T]$, $0\leq\Theta_{\delta}(s)\leq1,\;s\in[t,t+\delta]$. Because $(h^{\varepsilon_{n_i}}_{n_i},P^{\varepsilon_{n_i}}_{n_i})$ satisfies \eqref{eq:dd-h}, then we have
\begin{equation*}
\begin{array}{r@{}l}
& \displaystyle{\int^{T}_{0}\int\Theta_{\delta}(s)\Big[\partial_{s}u(s,x) h^{\varepsilon_{n_i}}_{n_i}(s,x) + \nabla u(s,x)\cdot P^{\varepsilon_{n_i}}_{n_i}(s,x)\Big]{\rm d}x{\rm d}s}\\ = & \displaystyle{ -\int^{t+\delta}_{t}\int\Theta'_{\delta}(s)u(s,x)h^{\varepsilon_{n_i}}_{n_i}(s,x){\rm d}x{\rm d}s - \int u(0,x)h^{\varepsilon_{n_i}}_{n_i}(0,x){\rm d}x}
\end{array}
\end{equation*}
Because $\forall s$, $\displaystyle{h^{\varepsilon_{n_i}}_{n_i}(s)\xrightarrow{w^{*}}h(s)}$, $P^{\varepsilon_{n_i}}_{n_i}\xrightarrow{w^{*}}P$ as $i\rightarrow \infty$ and the total variation of $h^{\varepsilon_{n_i}}_{n_i},\;P^{\varepsilon_{n_i}}_{n_i}$ is uniformly bounded, so by the weak-$*$ convergence and Lebesgue's dominated convergence theorem, let $i\rightarrow\infty$, we have
$$\int_0^T\int\Theta_{\delta}\Big(h\partial_{s}u + P\cdot\nabla u\Big)=- \int^{t+\delta}_{t}\int\Theta'_{\delta}(s)u(s)h(s) - \int u(0)h(0)$$
Now because $h\in C([0,T],C(\mathbb{T}^3,\mathbb{R})'_{w^{*}})$, so $\big\langle h(s),u(s)\big\rangle$ is a continuous function on $s$, then we let $\delta\rightarrow 0$, we have
$$-\int^{t+\delta}_{t}\Theta'_{\delta}(s)\big\langle h(s),u(s)\big\rangle{\rm d}s \longrightarrow \in u(t)h(t) $$
Because $\Theta_{\delta}\rightarrow\mathbbm{1}_{[0,t]}
$ for every $s\in [0,T]$, by Lebesgue's dominated convergence theorem, pass the limit $\delta\rightarrow 0$ on the left hand side, we finally get \eqref{eq:dddf-1}.

Moreover, because $\nabla\cdot B^{\varepsilon_{n_i}}_{n_i}(t)=0$. So for any $\phi\in C^{1}(\mathbb{T}^3,\mathbb{R})$, we have $\big\langle B^{\varepsilon_{n_i}}_{n_i}(t),\nabla\phi\big\rangle=0$. By taking the limit, we get $\big\langle B(t),\nabla\phi\big\rangle=0$. So \eqref{eq:dddf-2} is also satisfied.

At last, we will prove that $(h,B,D,P)$ satisfies \eqref{eq:dddf-3}. We first suppose that for fixed $N$, $0<h^{*}\in C^1([0,T]\times\mathbb{T}^3,\mathbb{R})$, $b^{*}\in C^1([0,T]\times\mathbb{T}^3,\mathbb{R}^3)$, $v^{*},d^{*}\in C^1([0,T],X_N)$ and $r$ is a big number such that $Q_r(w^*)$ is positive definite for all $t,x$. Here $Q$ is defined in \eqref{eq:ddm-1}. Now, let us denote
$$U=\big(\mathcal{L},B\big),\;U_i=\big(\mathcal{L},B^{\varepsilon_{n_i}}_{n_i}\big)\in C([0,T],C(\mathbb{T}^3,\mathbb{R}^{4})')$$
$$W=\big(U,D,P\big),\;W_i=\big(U_i,D^{\varepsilon_{n_i}}_{n_i},P^{\varepsilon_{n_i}}_{n_i}\big)\in L^{2}([0,T],C(\mathbb{T}^3,\mathbb{R}^{10})')$$
$$\widetilde{U}=U-Vh=\left(\mathcal{L}-h{h^{*}}^{-1},B-hb^{*}\right),\;\widetilde{U}_i=U_i-Vh^{\varepsilon_{n_i}}_{n_i}$$
$$\widetilde{W}=W-Fh{\rm d}s=\left(\widetilde{U}, D-hd^{*},P-hv^{*}\right),\;\widetilde{W}_i=W_i-Fh^{\varepsilon_{n_i}}_{n_i}$$
$$V=\big({h^{*}}^{-1},b^{*}\big),\;F=\big({h^{*}}^{-1},b^{*},d^{*},v^{*}\big)$$
$$E_i(t)=\int\frac{1}{2h^{\varepsilon_{n_i}}_{n_i}}\Big\{
(B^{\varepsilon_{n_i}}_{n_i}-h^{\varepsilon_{n_i}}_{n_i}b^{*})^2+\left(1-\frac{h^{\varepsilon_{n_i}}_{n_i}}{h^{*}}\right)^2
+\varepsilon_{n_i}\left[(D^{\varepsilon_{n_i}}_{n_i}-h^{\varepsilon_{n_i}}_{n_i}d^{*})^2+(P^{\varepsilon_{n_i}}_{n_i}
-h^{\varepsilon_{n_i}}_{n_i}v^{*})^2\right]\Big\}$$
Now, because $(h^{\varepsilon_{n_i}}_{n_i},B^{\varepsilon_{n_i}}_{n_i},v^{\varepsilon_{n_i}}_{n_i},d^{\varepsilon_{n_i}}_{n_i})$ is a some kind of ``solution'' to \eqref{eq:ddap-1}-\eqref{eq:ddap-4}, we have that for any $\varphi,\psi\in X_{n_i}$, any $g\in C^1(\mathbb{T}^3),\;\phi\in C^1(\mathbb{T}^3,\mathbb{R}^3)$,
\begin{equation*}
\int \partial_th^{\varepsilon_{n_i}}_{n_i}(t)g   - \int h^{\varepsilon_{n_i}}_{n_i}(t)v^{\varepsilon_{n_i}}_{n_i}(t)\cdot\nabla g =0
\end{equation*}
\begin{equation*}
\int \partial_t B^{\varepsilon_{n_i}}_{n_i}(t)\cdot\phi  - \int (B^{\varepsilon_{n_i}}_{n_i}\otimes v^{\varepsilon_{n_i}}_{n_i}-v^{\varepsilon_{n_i}}_{n_i}\otimes B^{\varepsilon_{n_i}}_{n_i}):\nabla\phi + \int d^{\varepsilon_{n_i}}_{n_i}\cdot(\nabla\times\phi)=0
\end{equation*}
\begin{multline*}
\int \partial_t\big(h^{\varepsilon_{n_i}}_{n_i}(t)v^{\varepsilon_{n_i}}_{n_i}(t)\big)\cdot\varphi -\int h^{\varepsilon_{n_i}}_{n_i}v^{\varepsilon_{n_i}}_{n_i}\otimes v^{\varepsilon_{n_i}}_{n_i}:\nabla\varphi -\int\big[(h^{\varepsilon_{n_i}}_{n_i} d^{\varepsilon_{n_i}}_{n_i}\cdot\nabla) d^{\varepsilon_{n_i}}_{n_i}\big]\cdot\varphi \\+ \int\nabla^lv^{\varepsilon_{n_i}}_{n_i}:\nabla^l\varphi   + \varepsilon_{n_i}^{-1}\int\left[{h^{\varepsilon_{n_i}}_{n_i}}^{-1}(B^{\varepsilon_{n_i}}_{n_i}\otimes B^{\varepsilon_{n_i}}_{n_i} +I_3):\nabla\varphi + h^{\varepsilon_{n_i}}_{n_i}v^{\varepsilon_{n_i}}_{n_i}\cdot\varphi\right]=0
\end{multline*}
\begin{multline*}
\int \partial_t\big(h^{\varepsilon_{n_i}}_{n_i}(t)d^{\varepsilon_{n_i}}_{n_i}(t)\big)\cdot\psi  - \int h^{\varepsilon_{n_i}}_{n_i}(d^{\varepsilon_{n_i}}_{n_i}\otimes v^{\varepsilon_{n_i}}_{n_i}-v^{\varepsilon_{n_i}}_{n_i}\otimes d^{\varepsilon_{n_i}}_{n_i}):\nabla\psi \\+\int\nabla^ld^{\varepsilon_{n_i}}_{n_i}:\nabla^l\psi  + \varepsilon_{n_i}^{-1}\int\big[-b^{\varepsilon_{n_i}}_{n_i}\cdot(\nabla\times\psi) + h^{\varepsilon_{n_i}}_{n_i}d^{\varepsilon_{n_i}}_{n_i}\cdot\psi \big]=0
\end{multline*}
For $n_i\geq N$, we can shoose $\phi=b^{\varepsilon_{n_i}}_{n_i}-b^*$, $\psi=\varepsilon_{n_i}\big(d^{\varepsilon_{n_i}}_{n_i}-d^*\big)$, $\varphi=\varepsilon_{n_i}\big(v^{\varepsilon_{n_i}}_{n_i}-v^*\big)$, and
$$g=\frac{1}{2}\Big[|h^{*}|^{-2}+|b^{*}|^2-|h^{\varepsilon_{n_i}}_{n_i}|^{-2}-|b^{\varepsilon_{n_i}}_{n_i}|^2+
\varepsilon_{n_i}\big(|v^{*}|^2+|d^{*}|^2-|v^{\varepsilon_{n_i}}_{n_i}|^2 -|d^{\varepsilon_{n_i}}_{n_i}|^2\big)
\Big]$$
With the specific chosen test function, we can get that (after a long progress of computation, we skip the tedious part here)
\begin{multline}\label{eq:ddpf-enin1}
\displaystyle{\frac{{\rm d}}{{\rm d}t}E_i(t) + \int\frac{\widetilde{W}_i^{{\rm T}}Q(w^*)\widetilde{W}_i}{2h^{\varepsilon_{n_i}}_{n_i}} + \int \widetilde{W}_i\cdot \mathrm{L}(w^*) -\varepsilon_{n_i}\widetilde{R}_i(t)}  \\ \displaystyle{= - \varepsilon_{n_i}\int\Big(
\big|\nabla^ld^{\varepsilon_{n_i}}_{n_i}\big|^2 + \big|\nabla^lv^{\varepsilon_{n_i}}_{n_i}\big|^2\Big) \leq 0}
\end{multline}
Here
\begin{equation*}
\begin{array}{r@{}l}
\widetilde{R}_i(t) &\displaystyle{=\int\frac{D^{\varepsilon_{n_i}}_{n_i}\otimes D^{\varepsilon_{n_i}}_{n_i} - P^{\varepsilon_{n_i}}_{n_i}\otimes P^{\varepsilon_{n_i}}_{n_i}}{2h^{\varepsilon_{n_i}}_{n_i}}:\big(\nabla v^{*}+{\nabla v^{*}}^{{\rm T}}\big)+ \int \frac{P^{\varepsilon_{n_i}}_{n_i}\otimes D^{\varepsilon_{n_i}}_{n_i}}{h^{\varepsilon_{n_i}}_{n_i}}:\big( \nabla d^{*}- {\nabla d^{*}}^{{\rm T}}\big)  }\\
& \displaystyle{  +\int\left(\nabla\left(\frac{|v^{*}|^2+|d^{*}|^2}{2}\right)-\partial_{t}v^{*}\right)\cdot P^{\varepsilon_{n_i}}_{n_i} -\int \partial_{t}d^{*}\cdot D^{\varepsilon_{n_i}}_{n_i} + \int h^{\varepsilon_{n_i}}_{n_i} \partial_{t}\left(\frac{|v^{*}|^2+|d^{*}|^2}{2}\right) }\\
& \displaystyle{ + \int\big[\nabla^{l}v^{\varepsilon_{n_i}}_{n_i}:\nabla^{l}v^{*}+\nabla^{l}d^{\varepsilon_{n_i}}_{n_i}:\nabla^{l}d^{*}\big] - \int D^{\varepsilon_{n_i}}_{n_i}\cdot\nabla\big(d^{\varepsilon_{n_i}}_{n_i}\cdot v^{*}\big) }
\end{array}
\end{equation*}
Then, for any $r$ such that $Q_r(w^*)>0$ for all $t,x$, we have
\begin{equation*}
\left(\frac{{\rm d}}{{\rm d}t}-r\right)E_i(t) + \int\frac{{\widetilde{W}_i}^{{\rm T}}Q_r(w^*)\widetilde{W}_i}{2h^{\varepsilon_{n_i}}_{n_i}} + \int \widetilde{W}_i\cdot \mathrm{L}(w^*)-\varepsilon_{n_i}\widetilde{R}_{i,r}(t)\leq 0
\end{equation*}
$$\widetilde{R}_{i,r}=\widetilde{R}_{i} - r\int\frac{h^{\varepsilon_{n_i}}_{n_i}}{2}\big[(d^{\varepsilon_{n_i}}_{n_i}-d^{*})^2+(v^{\varepsilon_{n_i}}_{n_i}-v^{*})^2\big]$$
So we have
\begin{equation*}
e^{-rt}E_i(t)+ \int^{t}_{0}e^{-rs}\left[\int\frac{{\widetilde{W}_i}^{{\rm T}}Q_r(w^*)\widetilde{W}_i}{2h^{\varepsilon_{n_i}}_{n_i}} + \int \widetilde{W}_i\cdot \mathrm{L}(w^*) -\varepsilon_{n_i}\widetilde{R}_{i,r}(s)\right]{\rm d}s \leq E_i(0)
\end{equation*}
Because $E_i(t)\geq\Lambda(h^{\varepsilon_{n_i}}_{n_i}(t),\widetilde{U}_i)$, then we have that
\begin{multline}\label{eq:ddpf-fenc1}
\displaystyle{e^{-rt}\Lambda(h^{\varepsilon_{n_i}}_{n_i}(t),\widetilde{U}_i(t)) +  \widetilde{\Lambda}(h^{\varepsilon_{n_i}}_{n_i},\widetilde{W}_i,e^{-rs}Q_r(w^*);0,t)}\\ \displaystyle{+ \int^{t}_{0}e^{-rs}\left(\int\widetilde{W}_i\cdot \mathrm{L}(w^*) - \varepsilon_{n_i}\widetilde{R}_{i,r}(s)\right){\rm d}s
\leq E_i(0)}
\end{multline}
Notice that $h^{\varepsilon_{n_i}}_{0}\xrightarrow{w^{*}}h_{0}$ in $C(\mathbb{T}^3,\mathbb{R})'$, and $$E_i(0)=\Lambda(h^{\varepsilon_{n_i}}_{0},U^{\varepsilon_{n_i}}_{0}) + \big\langle h^{\varepsilon_{n_i}}_{0}, \frac{1}{2}|V(0)|^2\big\rangle - \big\langle U^{\varepsilon_{n_i}}_{0},V(0)\big\rangle+ \varepsilon_{n_i}\big\langle h^{\varepsilon_{n_i}}_{0}, \frac{|v^{*}|^2+|d^{*}|^2}{2}\big\rangle$$
By \propref{prop:ddpf1}, we know that the right hand side of \eqref{eq:ddpf-fenc1} $E_i(0)\rightarrow \Lambda(h_{0},\widetilde{U}_{0})$ as $\varepsilon_{n_i}\rightarrow0$. By \lmref{lm:ddpf-cov1}, we know that $\sqrt{\varepsilon_{n_i}}\nabla^{l}v^{\varepsilon_{n_i}}_{n_i},\sqrt{\varepsilon_{n_i}}\nabla^{l}d^{\varepsilon_{n_i}}_{n_i}$ are uniformly bounded in $L^{2}_{t,x}$, thus $\sqrt{\varepsilon_{n_i}}v^{\varepsilon_{n_i}}_{n_i},\sqrt{\varepsilon_{n_i}}d^{\varepsilon_{n_i}}_{n_i}$ are uniformly bounded in $L^2(W^{1,\infty})$. Moreover $P^{\varepsilon_{n_i}}_{n_i},D^{\varepsilon_{n_i}}_{n_i},h^{\varepsilon_{n_i}}_{n_i}$ are uniformly bounded in $L^{2}_{t}(L^{1}_{x})$,  so we have that
$$\varepsilon_{n_i}\left|\int^{t}_{0}e^{r(t-s)}\widetilde{R}_{i,r}(s){\rm d}s\right|\leq\sqrt{\varepsilon_{n_i}}(1+\sqrt{\varepsilon_{n_i}})(1+|r|)C$$
Here $C$ only depends on $v^{*},b^{*},d^{*},h_{0},B_{0}$. So it goes to 0 as $\varepsilon_{n_i}\rightarrow0$. By the weak-* convergent of $P^{\varepsilon_{n_i}}_{n_i},D^{\varepsilon_{n_i}}_{n_i}\in C_{t,x}'$ and similar method as we did for \eqref{eq:dddf-1}, we have
$$\int_0^t\int e^{-rs}\widetilde{W}_i\cdot \mathrm{L}(w^*)\rightarrow \int_0^t\int e^{-rs}\widetilde{W}\cdot \mathrm{L}(w^*) $$
Besides, we have that,
$$
\liminf_{i\rightarrow\infty}\Lambda(h^{\varepsilon_{n_i}}_{n_i}(t),\widetilde{U}_i(t))  \geq \Lambda(h(t),\widetilde{U}(t))
$$
$$
\liminf_{i\rightarrow\infty}\widetilde{\Lambda}(h^{\varepsilon_{n_i}}_{n_i},\widetilde{W}_i,e^{-rs}Q_r(w^*);0,t)\geq \widetilde{\Lambda}(h,\widetilde{W},e^{-rs}Q_r(w^*);0,t)
$$
Combining the above results, we take the lower limit on both side of \eqref{eq:ddpf-fenc1}, then we can just get the inequality \eqref{eq:dddf-3} for all fixed $N$, $0<h^{*}\in C^1([0,T]\times\mathbb{T}^3,\mathbb{R})$, $b^{*}\in C^1([0,T]\times\mathbb{T}^3,\mathbb{R}^3)$, $v^{*},d^{*}\in C^1([0,T],X_N)$. Now for any $v^{*},d^{*}\in C^{1}([0,T]\times\mathbb{T}^3,\mathbb{R}^3)$ and $r$ such that $Q_r(w^*)$ is positive definite, because $v^{*},d^{*}$ is continuous, then there exist $r'<r$ such that $Q_{r'}(w^*)$ is still positive definite. Because $\bigcup_{n=1}^{\infty} C^1([0,T],X_n)$ is dense in $C^{1}([0,T]\times\mathbb{T}^3,\mathbb{R}^3)$, So we can find a sequence $\{v^{*}_{n}\},\{d^{*}_{n}\}\in C^1([0,T],X_n)$ that converge to $v^{*},d^{*}$ in $C^{1}([0,T]\times\mathbb{T}^3,\mathbb{R}^3)$ and $Q_r(w_n^*)$ is always positive definite, where $w_n^*=({h^*}^{-1},b^*,d_n^*,v_n^*)$. Now let us denote
$$\widetilde{W}_{n}=\widetilde{W}+\widetilde{F}_{n}h,\;\;\;\widetilde{F}_{n}=\big( 0,0,d^{*}-d^{*}_{n},v^{*}-v^{*}_{n}\big)$$
By Lebesgue's dominated convergence theorem, we have that
$$\int_0^t\int e^{-rs}\widetilde{W}_n\cdot \mathrm{L}(w^*_n)\rightarrow \int_0^t\int e^{-rs}\widetilde{W}\cdot \mathrm{L}(w^*) $$
Besides, we have
\begin{multline}\label{eq:pf-1}
\widetilde{\Lambda}(h,\widetilde{W}_{n},e^{-rs}Q_r(w_n^*);0,t)=\widetilde{\Lambda}(h,\widetilde{W},e^{-rs}Q_r(w_n^*);0,t) \\ + \int_{0}^{t}e^{-rs}\big\langle \widetilde{W},Q_{n}\widetilde{F}_{n}\big\rangle + \int_{0}^{t}e^{-rs}\big\langle h,\frac{1}{2}|\sqrt{Q_{n}}\widetilde{F}_{n}|^2\big\rangle
\end{multline}
Now we would like to take the limit $n\rightarrow\infty$. The following lemma will be useful:

\begin{lemma}
Suppose $Q_{n},Q\in C([0,T]\times\mathbb{T}^3,\mathbb{R}^{d^2})$ are positive definite, $\|Q_{n}- Q\|_{\infty}\rightarrow 0$ as $n\rightarrow\infty$, then $\displaystyle{ \liminf_{n\rightarrow\infty}\widetilde{\Lambda}(\rho,W,Q_n;0,t)\geq \widetilde{\Lambda}(\rho,W,Q;0,t) }$.
\end{lemma}

\begin{proof}
The proof is quite straightforward. For any $a\in C([0,T]\times\mathbb{T}^3,\mathbb{R}),A\in C([0,T]\times\mathbb{T}^3,\mathbb{R}^d)$ such that $a+\frac{1}{2}|\sqrt{Q^{-1}}A|^2\leq 0$, we have
$$\int_0^t\int\Big(a\rho + A\cdot W\Big)= \int_0^t\int\Big(\widetilde{a}\rho + A\cdot W\Big) + \int_0^t\int(a-\widetilde{a})\rho$$
where
$$\widetilde{a}= a + \frac{1}{2}|\sqrt{Q^{-1}}A|^2-\frac{1}{2}|\sqrt{Q_{n}^{-1}}A|^2 $$
Since $\widetilde{a} + \frac{1}{2}|\sqrt{Q_{n}^{-1}}A|^2\leq 0$, we have
$$
\int_0^t\int\Big(a\rho + A\cdot W\Big) \leq \widetilde{\Lambda}(\rho,W,Q_n;0,t) + \frac{1}{2}\|\rho\|_{TV}\left\||\sqrt{Q_{n}^{-1}}A|^2-|\sqrt{Q^{-1}}A|^2\right\|_{\infty}
$$
Take the lower limit on both sides, we have, for any $(a,A)$ s.t. $a+\frac{1}{2}|\sqrt{Q^{-1}}A|^2\leq 0$,
$$\int_0^t\int\Big(a\rho + A\cdot W\Big) \leq  \liminf_{n\rightarrow\infty}\widetilde{\Lambda}(\rho,W,Q_n;0,t)$$
So we have $\displaystyle{ \widetilde{\Lambda}(\rho,W,Q;0,t)\leq  \liminf_{n\rightarrow\infty}\widetilde{\Lambda}(\rho,W,Q_n;0,t)}$.
\end{proof}

Now, by taking the lower limit as $n\rightarrow\infty$ in \eqref{eq:pf-1}, we can get that \eqref{eq:dddf-3} is valid for $C^{1}$ functions. So we have completely proved the existence of a dissipative solution. We summarize our result in the following theorem.

\begin{theorem}
Suppose that $B_{0}\in C(\mathbb{T}^3,\mathbb{R}^3)',\;h_{0}\in C(\mathbb{T}^3,\mathbb{R})'$, satisfying that $\nabla\cdot B_{0}=0$ in the sense of distributions and $\Lambda(h_{0},U_{0})<\infty$, where $U_{0}=(\mathcal{L},B_{0})$. Then there exists a dissipative solution $(h,B,D,P)$ of (DMHD) with initial value $(h,B)|_{t=0}=(h_{0},B_{0})$.

\end{theorem}

\section{Appendix: proof of \lmref{lm:dd-entropy}}

Let us consider a more general case where $(h,B,D,P)$ only satisfies the continuity equation \eqref{eq:dd-h} and the divergence-free constraint \eqref{eq:dd-div}. We denote
$$
\phi=\partial_{t}B +\nabla\times\left(\frac{D + B \times P}{h}\right),
$$
$$
\psi=D - \nabla\times\left(\frac{B}{h}\right)
,\;\;\;
\varphi=P - \nabla\cdot\left(\frac{B\otimes B}{h}\right)-\nabla\left(\frac{1}{h}\right),
$$
Note that $\phi,\psi,\varphi$ vanish when $(h,B,D,P)$ is exactly a solution to the Darcy MHD \eqref{eq:dd-h}-\eqref{eq:dd-div}. We also use the non-conservative variables, namely,
$$
\tau=\frac{1}{h},\;b=\frac{B}{h},\;d=\frac{D}{h},\;v=\frac{P}{h}
$$
and, for the convenience of writing, let's denote
$$U=\big(1,B\big),\;\;u^*=({h^*}^{-1},b^*),\;\;W=(1,B,D,P),\;\;w^*=({h^*}^{-1},b^*,d^*,v^*).
$$

To prove the lemma, let's start with computing the time derivative of the energy
$$
S(t)=\int_{\mathbb{T}^3}\frac{|U|^2}{2h}=\int_{\mathbb{T}^3}\frac{1+B^2}{2h}
$$
Quite similar to \eqref{eq:dde1}, we have
\begin{equation}\label{eq:apdx-1}
\begin{array}{r@{}l}
\displaystyle{S'(t)} & = \displaystyle{\int b\cdot\partial_{t}B  -\int\frac{1}{2}(\tau^2+ b^2)\partial_{t}h}\\
& = \displaystyle{ \int b\cdot\big[\phi-\nabla\times(B\times v +d)\big] + \int\frac{1}{2}(\tau^2+ b^2)\nabla\cdot P  }\\
& = \displaystyle{ \int b\cdot\phi-\int (B\times v +d)\cdot(\nabla\times b)-\int \big[\tau\nabla\tau + \nabla (b^2/2)\big]\cdot P  }\\
& = \displaystyle{ \int b\cdot\phi-\int d\cdot(\nabla\times b)-\int v\cdot\big(\nabla\cdot(b\otimes B) + \nabla\tau\big)  }\\
& = \displaystyle{\int\big( b\cdot\phi + d\cdot \psi + v\cdot\varphi\big)-\int\frac{D^2+P^2}{h}}
\end{array}
\end{equation}
Now, let's look at the relative entropy. Since $\widetilde{U}=U-hu^*$, we have
\begin{equation*}
\int\frac{\big|\widetilde{U}\big|^2}{2h}=S(t) + \int\frac{h\big|u^*\big|^2}{2} -\int U\cdot u^*
\end{equation*}
Therefore, we have
\begin{equation}\label{eq:apdx-2}
\begin{array}{r@{}l}
\displaystyle{\frac{{\rm d}}{{\rm d}t}\int\frac{\big|\widetilde{U}\big|^2}{2h} }
& = \displaystyle{S'(t) + \int\frac{\partial_t h}{2}\big|u^*\big|^2 + \int h u^*\cdot\partial_t u^* -\int\partial_tU\cdot u^* -\int U\cdot\partial_t u^* } \\
& = \displaystyle{ S'(t)- \int \frac{\nabla\cdot P}{2}\big|u^*\big|^2 - \int\partial_t B\cdot b^* -\int\widetilde{U}\cdot\partial_t u^* }\\
& \displaystyle{ = S'(t) + \int \frac{P}{2}\cdot\nabla\big|u^*\big|^2 - \int b^*\cdot\big[\phi-\nabla\times(B\times v +d)\big] -\int \widetilde{U}\cdot\partial_t u^*}
\end{array}
\end{equation}
Now let's use a small trick to write 0 as,
\begin{equation*}
\begin{array}{r@{}l}
\displaystyle{0\; }
& = \displaystyle{ \int \Big[d^*\cdot(D-\nabla\times b-\psi) + v^*\cdot(P-\nabla\cdot(b\otimes B)-\nabla\tau-\varphi)\Big]} \\
& = \displaystyle{ \int \big(D\cdot d^* +P\cdot v^*-d^*\cdot \psi-v^*\cdot\varphi\big) + \int\left[\sum_{i,j=1}^3\frac{\partial_jv^*_i}{h}B_iB_j + \frac{\nabla\cdot v^*}{h} -\frac{\big(\nabla\times d^*\big)\cdot B}{h}  \right] }
\end{array}
\end{equation*}
Then by \eqref{eq:apdx-1},\eqref{eq:apdx-2}, we have,
\begin{equation}
\begin{array}{r@{}l}
\displaystyle{\frac{{\rm d}}{{\rm d}t}\int\frac{\big|\widetilde{U}\big|^2}{2h} }
& = \displaystyle{ \sum_{i,j=1}^3\int\Big[\frac{\partial_jv^*_i}{h}B_iB_j- \frac{\partial_jb^*_i-\partial_ib^*_j}{h}B_iP_j \Big] -\int\frac{D^2+P^2}{h} + \int\frac{\nabla\cdot v^*}{h} }\\
& \displaystyle{ \qquad + \int\Big[\frac{\big(\nabla\times b^*\big)\cdot D}{h}-\frac{\big(\nabla\times d^*\big)\cdot B}{h}\Big] + \int \Big[D\cdot d^* +P\cdot \Big(v^*+\frac{\nabla|u^*|^2}{2}\Big)\Big]}\\
& \displaystyle{ \qquad - \int U\cdot\partial_t u^* + \int\Big[\phi\cdot \big(b-b^*\big) +\psi\cdot \big(d-d^*\big) + \varphi\cdot \big(v-v^*\big) \Big] }\\
& \displaystyle{ = - \int\frac{W^{\rm T}Q(w^*)W}{2h}  -\int \widetilde{U}\cdot\partial_t u^* + \int \Big[D\cdot d^* +P\cdot \Big(v^*+\frac{\nabla|u^*|^2}{2}\Big)\Big] }\\
& \displaystyle{    \qquad + \int\Big[\phi\cdot \big(b-b^*\big) +\psi\cdot \big(d-d^*\big) + \varphi\cdot \big(v-v^*\big) \Big] }
\end{array}
\end{equation}
Now since $W=\widetilde{W}+hw^*$, we can rewrite the quadratic like term as
\begin{equation*}
\begin{array}{r@{}l}
\displaystyle{ \int\frac{W^{\rm T}Q(w^*)W}{2h}  } & \displaystyle{ = \int\frac{\big(\widetilde{W}+hw^*\big)^{\rm T}Q(w^*)\big(\widetilde{W}+hw^*\big)}{2h} }\\
& \displaystyle{ = \int\left(\frac{\widetilde{W}^{\rm T}Q(w^*)\widetilde{W}}{2h} + \widetilde{W}\cdot Q(w^*)w^* + \frac{{w^*}^{\rm T}Q(w^*)w^*}{2}h\right) }
\end{array}
\end{equation*}
A direct computation gives that
\begin{equation*}
Q(w^*)w^*=\mathrm{L}(w^*)-(\partial_t u^*,0,0) + \left(\nabla\cdot\big(d^*\times b^*-{h^*}^{-1}v^*\big), -\nabla\big(b^*\cdot v^*\big),d^*,v^*+\frac{1}{2}\nabla\big|u^*\big|^2\right)
\end{equation*}
\begin{equation*}
\frac{{w^*}^{\rm T}Q(w^*)w^*}{2}=\frac{v^*}{2}\cdot\nabla\big|u^*\big|^2-
b^*\cdot\nabla\big(b^*\cdot v^*\big) + {h^*}^{-1}\nabla\cdot\big(d^*\times b^*-{h^*}^{-1}v^*\big)+{d^*}^2+{v^*}^2
\end{equation*}
Therefore, since $B$ is divergence free, we have
\begin{equation*}
\begin{array}{r@{}l}
\displaystyle{ \int\frac{W^{\rm T}Q(w^*)W}{2h}  } & \displaystyle{ = \int\left[\frac{\widetilde{W}^{\rm T}Q(w^*)\widetilde{W}}{2h} + \widetilde{W}\cdot \mathrm{L}(w^*)-\widetilde{U}\cdot\partial_t u^* + D\cdot d^* +P\cdot \Big(v^*+\frac{\nabla|u^*|^2}{2}\Big) \right] }\\
\end{array}
\end{equation*}
So, finally, we have
\begin{equation}
\frac{{\rm d}}{{\rm d}t}\int\frac{\big|\widetilde{U}\big|^2}{2h} + \int\frac{W^{\rm T}Q(w^*)W}{2h} +\int\widetilde{W}\cdot\mathrm{L}(w^*) = \int\Big[\phi\cdot \big(b-b^*\big) +\psi\cdot \big(d-d^*\big) + \varphi\cdot \big(v-v^*\big) \Big]
\end{equation}
Especially, when $(h,B,D,P)$ is a solution of the Darcy MHD \eqref{eq:dd-h}-\eqref{eq:dd-div}, i.e., $\psi=\phi=\varphi=0$, we obtain \eqref{eq:dd-rel-1}. Moreover, if $(h^*,h^*b^*,h^*d^*,h^*v^*)$ is also a solution of \eqref{eq:dd-h}-\eqref{eq:dd-div}, it is quite easy to verify that $\mathrm{L}_{h}(w^*),\mathrm{L}_{B}(w^*),\mathrm{L}_{D}(w^*),\mathrm{L}_{P}(w^*)$ respectively correspond to the equation for the non-conservative variables $(\tau,b,d,v)$, thus vanish.

\end{document}